\documentclass{amsart}

\usepackage{amssymb}
\usepackage{amsmath}
\usepackage{amsthm}
\usepackage{amscd,mdwlist}
\usepackage{hyperref}
\usepackage{xcolor}
\usepackage{cite}
\usepackage{bbm}
\usepackage{graphicx}
\usepackage{subcaption}
   
    \theoremstyle{plain}
\newtheorem{theorem}{Theorem}[section]
\newtheorem{proposition}{Proposition}[section]
\newtheorem{corollary}{Corollary}[section]
\newtheorem{lemma}{Lemma}[section]

\theoremstyle{remark}
\newtheorem{remark}{Remark}[section]
\newtheorem{examples}{Examples}[section]
\newtheorem{assumption}{Assumption}[section]

\numberwithin{equation}{section}

\DeclareMathOperator{\supp}{supp}

\DeclareMathOperator{\diverg}{div}

\DeclareMathOperator{\loc}{loc}

\begin{document}

\title[Self-adjoint Laplacians on hyperbolic attractors]{Self-adjoint Laplacians and symmetric diffusions on hyperbolic attractors}

\author{Shayan Alikhanloo$^1$, Michael Hinz$^2$}
\thanks{$^1$, $^2$ Research supported by the DFG IRTG 2235: 'Searching for the regular in the irregular: Analysis of singular and random systems'.}
\address{$^1$Fakult\"at f\"ur Mathematik, Universit\"at Bielefeld, Postfach 100131, 33501 Bielefeld, Germany}
\email{salikhan@math.uni-bielefeld.de}
\address{$^2$Fakult\"at f\"ur Mathematik, Universit\"at Bielefeld, Postfach 100131, 33501 Bielefeld, Germany}
\email{mhinz@math.uni-bielefeld.de}

\begin{abstract} 
We construct self-adjoint Laplacians and symmetric Markov semigroups on hyperbolic attractors, endowed with Gibbs $u$-measures. If the measure has full support, we can also conclude the existence of an associated symmetric diffusion process. In the special case of partially hyperbolic diffeomorphisms induced by geodesic flows on negatively curved manifolds the Laplacians we consider are self-adjoint extensions of well-known classical leafwise Laplacians. We observe a quasi-invariance property of energy densities in the $u$-conformal case and the existence of nonconstant functions of zero energy.
\tableofcontents
\end{abstract}


\maketitle

\section{Introduction}

We consider partially hyperbolic attractors, \cite{HaPe,Barreira, BarreiraPesin, BrinStuck, Pesin,Pesin1, Robinson}, hyperbolic attractors with singularities, \cite{CLP17, Pesin92}, and attractors with nonuniformly hyperbolic structure, \cite{BarreiraPesin, BarreiraPesin2}, endowed with Gibbs $u$-measures, \cite{CLP17, LedrappierYoung1, Pesin, PesinSinai, Young}. We construct Laplacians and diffusions that act respectively move leafwise in the unstable directions and are self-adjoint respectively symmetric with respect to the given Gibbs $u$-measure. 

On smooth manifolds, graphs and certain groups and metric measure spaces symmetric diffusions and self-adjoint Laplacians are well-understood and provide additional insight into structural features of the space. In terms of geometric complexity, attractors of dynamical systems are more involved than manifolds, but they still display more features of smoothness than fractals. Analysis on manifolds and on certain classes of fractals is well-developed (see for instance \cite{Da89, Grigoryanbook, Kigami01} and the references cited there), but an analysis on hyperbolic attractors is yet to be explored. On the other hand hyperbolic attractors have natural tangential structures, unambigously determined by the stable manifold theorem, \cite{BarreiraPesin, BrinStuck, Robinson, Shub}, and they carry distinguished volume measures that capture the dynamics of the system in an optimal way, namely Gibbs $u$-measures, and in particular, SRB measures. See \cite{Bowen74, Bowen75, CDP16, CLP17, Ledrappier, LedrappierYoung1, PesinSinai, Ruelle76, Sinai68, Sinai72, Young}.

Expanding hyperbolic attractors in Williams' sense, \cite{Williams}, are, roughly speaking, topologically conjugate to abstract solenoids, see Remark \ref{R:expanding} below. These abstract solenoids are foliated spaces in the sense of \cite{Candel03, CandelConlonI, CandelConlonII}, and the analysis developed there can be applied. Alternatively, they can be studied from group theoretic point of view, cf. \cite[(10.12) Definition and (10.13) Theorem]{HewittRoss}, the related harmonic analysis is classical. Recent results on analysis on inverse limit spaces can be found in \cite{AlonsoRuiz18, AlonsoRuiz21, Steinhurst}. A theory of metric spaces with an attractor-like structure had been established in \cite{Putnam}. The existence results for self-adjoint Laplacians and symmetric diffusions we provide here are new, and they are of a different type: Since they are based on the stable manifold theorem and the concept of Gibbs $u$-measures, they are results in the differentiable category rather than in the topological. Somewhat related studies of symmetric diffusions on pattern spaces, fractals invariant under Kleinian groups and on certain repellers (Julia sets) can be found in \cite{ARHTT, Kajino20a, Kajino20b} respectively \cite{RogersTeplyaev10}.

We use a Dirichlet form approach, \cite{BH91, Da89, FOT94}, which seems well-adapted: If the given diffeomorphism is $C^{1+\alpha}$,  the local unstable manifolds are $C^1$, so that one cannot talk about $C^2$-functions in leafwise direction. Even if the diffeomorphism is $C^{r+\alpha}$ with $r\geq 2$, the densities of the conditional measures on the local unstable manifolds are typically H\"older, a regularity too low to introduce classical Laplacians. On the other hand the $C^1$-regularity of the unstable manifolds allows to define quadratic forms, which then give rise to well-defined operators. The situation is similar to that encountered for second order elliptic operators with bounded measurable coefficients in divergence form, which are typically studied using Dirichlet forms. 

In Section \ref{S:Laplacians} we introduce an abstract setup, general enough to accommodate partially hyperbolic attractors and hyperbolic attractors with singularities; the basic structure is fixed in Assumption \ref{A:A1}. Within this general setup we can use leafwise gradients, defined in the classical sense, to introduce quadratic forms of Dirichlet integral type associated with a given finite Borel measure. Under suitable measurability conditions on the tangential structure and absolute continuity conditions on conditional measures (Assumptions \ref{A:A2} and \ref{A:A3}) this quadratic form extends to a local Dirichlet form, and we may regard its generator as a self-adjoint Laplacian on the space, Theorem \ref{T:closable}. Related first order quantities extend accordingly, Corollary \ref{C:firstorder}, and the unique symmetric Markov semigroup corresponding to the Dirichlet form is recurrent and conservative, Corollary \ref{C:semigroup}. Under additional assumptions the Dirichlet form is strongly local and there is a symmetric Hunt diffusion process uniquely associated with it, Theorem \ref{T:regular}. In Section \ref{S:hyperbolic} we first collect well-known facts on partially and uniformly hyperbolic attractors and Gibbs $u$- (and SRB) measures and then observe that the results from Section \ref{S:Laplacians} apply and yield natural self-adjoint Laplacians and symmetric diffusions on these attractors, Theorem \ref{T:pha}. In Section \ref{S:geodesicflow} we discuss the case of geodesic flows on negatively curved manifolds, \cite{Anosov67, BarreiraPesin, BrinStuck, BG05, FH19, H17, KatokHasselblatt, Poll93}. In this special case our results recover a well-known construction of leafwise Laplacians and diffusions, see for instance \cite{Yue, Yue95}, and they may be viewed as an extension of this construction to the attractors of dissipative systems. In Section \ref{S:quasiinv} we consider the $u$-conformal case and observe an interesting quasi-invariance property for the energy densities, Theorem \ref{T:quasiinv}, it contrasts the Cameron-Martin theorem, see Remark \ref{R:CM}. A result on the existence of nonconstant functions of zero energy is provided by Theorem \ref{T:noLiouville} in Section \ref{S:zero}, it roughly speaking says that in general Liouville or foliated Liouville properties do not hold. Our method also yields self-adjoint Laplacians on hyperbolic attractors with singularities and on attractors with nonuniformly hyperbolic structure, examples are attractors of Lorenz, Lozi or Belykh type, see \cite{Belykh, Levy, Lozi, Misiurewicz, Pesin92, Sataev99}, and the H\'enon attractor, \cite{BarreiraPesin2, BeCa91, BeY93, Henon}. This is shown in Sections \ref{S:has} and \ref{S:nonuni}. To make the article as self-contained as possible, we survey basic facts on Rokhlin's theorem in Appendix \ref{S:Rokhlin}, provide a brief sketch of the geometric approach to Gibbs $u$-measures in Appendix \ref{S:SRB},  discuss weighted manifolds of low regularity in Appendix \ref{S:manifolds} and sketch a known superposition argument for quadratic forms in Appendix \ref{S:superpos}.

\section*{Acknowledgements}

The authors wish to thank Vaughn Climenhaga, Steven Frankel, Alexander Grigoryan, Marc Pollicott and Alexander Teplyaev for helpful discussions and comments.

\section{Dirichlet forms and self-adjoint Laplacians}\label{S:Laplacians}

We consider an abstract setup and introduce quadratic forms that are analogs of the classical Dirichlet integral and extend to local Dirichlet forms. 

Let $M$ be a smooth Riemannian manifold and $A\subset M$ a Borel set. We make several structural assumptions on $A$.

\begin{assumption}\label{A:A1}
We assume that $r\geq 1$ and that $A$ admits a partition into immersed submanifolds of $M$ of class $C^r$. Given $x\in A$, we write $W(x)$ for the partition element containing $x$. Along with $x$, $W(x)$ contains a set $B_x$ which is open in $W(x)$ and locally compact in the topology on $A$ induced by that of $M$.
\end{assumption}


We note for later use that under Assumption \ref{A:A1} the set $A$ is locally compact and separable in the topology induced by that of $M$.

By $C(A)$ we denote the space of continuous functions on $A$ (with respect to the subspace topology), by $C(M)|_{A}$ and 
$C^k(M)|_{A}$, $1\leq k\leq r$, we denote the spaces of restrictions to $A$ of functions from $C(A)$ and $C^k(M)$, respectively.
If $A$ is closed, then $C(M)|_{A}=C(A)$ by Tietze's extension theorem. We write $C^u(A)$ (resp. $C^{u,k}(A)$) for the space of Borel functions $\varphi:A\rightarrow \mathbb{R}$ whose restriction to any $B_x$, $x\in A$, is a continuous (resp. $C^k$-) function on $B_x$. Clearly $C^{k+1}(M)|_{A}\subset C^{k}(M)|_{A}$ and $C^{u,k+1}(A)\subset C^{u,k}(A)$, and the spaces form algebras under pointwise multiplication. 
\begin{proposition}\label{P:naive}
For any $1\leq k\leq r$ we have $C^k(M)|_{A}\subset C^{u,k}(A)$, in particular, $C(A)\subset C^{u}(A)$.
\end{proposition}

\begin{proof}
For any $x\in A$ the pointwise restriction of a function $g\in C^k(M)$ is in $C^k(W(x))$  because $W(x)$ is a $C^r$ immersed submanifold, \cite[Theorem 5.27]{Lee}.
\end{proof}

Let each of the immersed submanifolds $W(x)$ be endowed with the Riemannian metric inherited from $M$. We write $P_{T_xW(x)}$ to denote the orthogonal projection in $T_xM$ onto $T_xW(x)$. By $\nabla_M$ and $\nabla_{W(x)}$ we denote the gradient operators on $M$ and on $W(x)$, respectively. Assumption \ref{A:A1} ensures that for any $x\in A$ and any $\varphi\in C^{u,1}(A)$ the gradient $\nabla_{W(x)}\varphi$ of $\varphi$ is well-defined on $B_x\subset A\cap W(x)$. 

\begin{proposition}\label{P:restrict}
For any function $g\in C^1(M)$ and any $x\in A$ we have 
\[\nabla_{W(x)} g|_{A}(x)=P_{T_xW(x)}(\nabla_M g(x))\] 
in $T_xW(x)$ and 
\begin{equation*}\label{E:gradientineq}
\left\|\nabla_{W(x)} g|_{A}(x)\right\|_{T_xW(x)}\leq \left\|\nabla_M g(x)\right\|_{T_xM}.
\end{equation*}
\end{proposition}

\begin{proof}
By definition the gradient $\nabla_{W(x)} g|_{A}(x)$ on $W(x)$ of $g|_{A}$ at $x$ is the unique element of $T_xW(x)$ such that 
$v(g|_{A})=\big\langle \nabla_{W(x)} g|_{A}(x), v\big\rangle_{T_xW(x)}$
for all $v\in T_xW(x)$. For any such $v$ we can find an open interval $I\subset\mathbb{R}$ around zero and a $C^1$ curve $\gamma:I\to W(x)\subset M$ with $\gamma(0)=x$ and $\dot{\gamma}(0)=v$, and we have  
\[v(g|_{A})=\frac{d}{dt}g(\gamma(t))|_{t=0}=v(g).\]
For the gradient $\nabla_M g(x)$ on $M$ of $g$ at $x$ we have, again by definition, 
\begin{multline}
v(g)=\big\langle\nabla_M g(x),v\big\rangle_{T_xM}=\big\langle P_{T_xW(x)}(\nabla_M g(x)),v\big\rangle_{T_xM}\notag\\
=\big\langle P_{T_xW(x)}(\nabla_M g(x)),v\big\rangle_{T_xW(x)},
\end{multline}
because $v\in T_xW(x)$. Since $v$ was arbitrary, this implies the statement.
\end{proof}
 
Assumption \ref{A:A1} allows to introduce a 'classical' leafwise gradient operator $\nabla$ on $C^{u,1}(A)$ by 
\begin{equation}\label{E:classgrad}
\nabla \varphi(x):=\nabla_{W(x)}\varphi(x), \quad x\in A,\quad \varphi\in C^{u,1}(A).
\end{equation}

Let $\mu$ be a finite Borel measure on $A$. We define a quadratic form $(\mathcal{E},\overline{\mathcal{D}}_0(\mathcal{E}))$ by
\begin{equation}\label{E:initialdomain}
\overline{\mathcal{D}}_0(\mathcal{E}):=\left\lbrace \varphi\in L^2(A,\mu)\cap C^{u,1}(A):\ x\mapsto \left\|\nabla \varphi(x)\right\|_{T_xW(x)}\ \text{is in $L^2(A,\mu)$}\right\rbrace
\end{equation}
and 
\begin{equation}\label{E:defquadform}
\mathcal{E}(\varphi):=\int_{A} \left\|\nabla\varphi (x)\right\|_{T_xW(x)}^2\:\mu(dx),\quad \varphi\in \overline{\mathcal{D}}_0(\mathcal{E}).
\end{equation}

The form $\mathcal{E}$ is an analog of the classical Dirichlet integral. By polarization it may be seen as a nonnegative definite symmetric bilinear form on $\overline{\mathcal{D}}_0(\mathcal{E})$.

\begin{remark}
Obviously $(\mathcal{E},\overline{\mathcal{D}}_0(\mathcal{E}))$ depends on the choice of $\mu$ and should rather be denoted by $(\mathcal{E}^{(\mu)},\overline{\mathcal{D}}_0(\mathcal{E}^{(\mu)}))$. However, since $\mu$ will be fixed, we omit it from notation, and we will follow the same agreement for other items.
\end{remark}

We also use the notation 
\begin{equation}\label{E:secdom}
\mathcal{D}_0(\mathcal{E}):=\overline{\mathcal{D}}_0(\mathcal{E})\cap C_c(A).
\end{equation}

\begin{remark} 
Different choices of a priori domains (\ref{E:initialdomain}) lead to different forms and Laplacians below. We concentrate on (\ref{E:initialdomain}), because it is large from a 'transversal' point of view, and on (\ref{E:secdom}), because it is well-connected to the topology of $A$ and one can use the theory in \cite{FOT94}.
\end{remark}

\begin{assumption}\label{A:A2}
We  assume that for any  continuous vector field $v$ on $M$ the map $x\mapsto \left\|P_{T_xW(x)}v(x)\right\|_{T_xM}$
is Borel measurable on $A$. 
\end{assumption}
Then by Proposition \ref{P:restrict} we have $C^{1}_c(M)|_{A}\subset \mathcal{D}_0(\mathcal{E})$. Viewing $\mu$ as a measure on $M$, also the following density statement becomes obvious. 

\begin{proposition}\label{P:dense}
For any finite Borel measure $\mu$ on $A$ and any nonnegative integer $k\leq r$ the space $C^k_c(M)|_{A}$ is a dense subspace of $L^2(A,\mu)$.
\end{proposition}

We call a Borel subset $\mathcal{R}\subset A$ an \emph{abstract rectangle} if it has positive measure $\mu(\mathcal{R})>0$ and admits a measurable partition into open subsets $B$ of the submanifolds $W(x)$, $x\in A$. 

\begin{assumption}\label{A:A3}
We assume that $A=\bigcup_{\ell\geq 1} A_\ell$ with compact subsets $A_\ell$ of $M$ such that $A_1\subset A_2\subset \dots$ and that for any $\ell\geq 1$ the set $A_\ell$ can be covered by finitely many abstract rectangles $\mathcal{R}_{\ell,j}\subset A_\ell$, $j=1,...,n_\ell$.
\end{assumption}

On each abstract rectangle $\mathcal{R}_{\ell,j}$ the disintegration identity
\begin{equation}\label{E:disintegration}
\mu(E)=\int_{\mathcal{P}_{\ell,j}} \, \mu_{B}(E)\, \hat{\mu}_{\mathcal{P}_{\ell,j}}(dB), \quad \text{$E\subset \mathcal{R}_{\ell,j}$ Borel,}
\end{equation}
holds, where $\hat{\mu}_{\mathcal{P}_{\ell,j}}$ is the pushforward of $\mu$ under the canonical projection onto the elements $B$ of the partition $\mathcal{P}_{\ell,j}$ and for each $B\in \mathcal{P}_{\ell,j}$ the symbol $\mu_B$ denotes the conditional measure on $B$. See Appendix \ref{S:Rokhlin}.

We write $m_{W(x)}$ for the Riemannian volume on $W(x)$. Given $\ell$, $j$ and $B\in \mathcal{P}_{\ell,j}$ we denote the restriction to $B$ of $m_{W(x)}$ by $m_B$. We say that $\mu$ satisfies the\emph{ (AC)-property} if for any abstract rectangle $\mathcal{R}_{\ell,j}$ the conditional measures $\mu_{B}$ in (\ref{E:disintegration}) are absolutely continuous with respect to $m_{B}$ for $\hat{\mu}_{\mathcal{P}_{\ell,j}}$-a.e. $B$. By $\mathcal{M}^{ac}(A)$ we denote  the set of all Borel probability measures $\mu$ on $A$ satisfying the (AC)-property and having bounded, strictly positive and continuous Radon-Nikodym densities $d\mu_B/dm_B$.

We extend $\mathcal{E}$ to a standard functional analysis setup, cf. \cite[Section VIII.6]{RS80}. Recall that a pair $(\mathcal{Q},\mathcal{D}(\mathcal{Q}))$ is said to be a \emph{densely defined nonnegative definite symmetric bilinear form} on the Hilbert space $L^2(A,\mu)$ if $\mathcal{D}(\mathcal{Q})$ is a dense subspace of $L^2(A,\mu)$ and $\mathcal{Q}$ is a nonnegative definite symmetric bilinear form on $\mathcal{D}(\mathcal{Q})$. If in addition $\mathcal{D}(\mathcal{Q})$, endowed with the scalar product $(\varphi,\psi)\mapsto \mathcal{Q}(\varphi,\psi)+\left\langle \varphi,\psi\right\rangle_{L^2(A,\mu)}$, is a Hilbert space, then we call $(\mathcal{Q},\mathcal{D}(\mathcal{Q}))$ a \emph{closed quadratic form}. To each closed quadratic form $(\mathcal{Q},\mathcal{D}(\mathcal{Q}))$ on $L^2(A,\mu)$ there corresponds a unique non-positive definite self-adjoint operator $(\mathcal{T},\mathcal{D}(\mathcal{T}))$ on $L^2(A,\mu)$ such that $\mathcal{Q}(\varphi,\psi)=-\left\langle \mathcal{T}\varphi,\psi\right\rangle_{L^2(A,\mu)}$ for all $\varphi\in \mathcal{D}(\mathcal{T})$ and $\psi\in \mathcal{D}(\mathcal{Q})$; it is called the \emph{generator} of $(\mathcal{Q},\mathcal{D}(\mathcal{Q}))$. See \cite[Theorem 1.3.1]{FOT94} or \cite[Theorem VIII.15]{RS80}. If $\mathcal{D}_0(\mathcal{Q})$ is a dense subspace of $L^2(A,\mu)$ and $\mathcal{Q}$ is a nonnegative definite symmetric bilinear form on $\mathcal{D}_0(\mathcal{Q})$, then $(\mathcal{Q}, \mathcal{D}_0(\mathcal{Q}))$ is said to be \emph{closable} if there is a closed quadratic form $(\mathcal{Q}',\mathcal{D}(\mathcal{Q}'))$ such that $\mathcal{D}_0(\mathcal{Q})\subset \mathcal{D}(\mathcal{Q}')$ and $\mathcal{Q}'=\mathcal{Q}$ on $\mathcal{D}_0(\mathcal{Q})$. This is the case if for any sequence $(\varphi_j)_j\subset \mathcal{D}_0(\mathcal{Q})$ that is Cauchy with respect to the seminorm $\mathcal{Q}^{1/2}$ and such that $\lim_j \varphi_j=0$ in $L^2(A,\mu)$ we have $\lim_j \mathcal{Q}(\varphi_j)=0$. A \emph{Dirichlet form} $(\mathcal{Q},\mathcal{D}(\mathcal{Q}))$ on $L^2(A,\mu)$ is a closed quadratic form satisfying the Markov property, i.e. such that $u\in\mathcal{D}(\mathcal{Q})$ implies $u\wedge 1 \in \mathcal{D}(\mathcal{Q})$ and $\mathcal{Q}(u\wedge 1)\leq \mathcal{Q}(u)$, see for instance \cite[Chapter I, 1.1.1 and 3.3.1]{BH91} or \cite{Da89, FOT94}. If $\mathbf{1}\in\mathcal{D}(\mathcal{Q})$ then the Dirichlet form 
$(\mathcal{Q},\mathcal{D}(\mathcal{Q}))$ is said to be \emph{local} if for any $F,G\in C^\infty(\mathbb{R})$ with disjoint supports and any $u\in\mathcal{D}(\mathcal{Q})$ we have $\mathcal{Q}(F(u),G(u))=0$, \cite[Chapter I, Corollary 5.1.4]{BH91}.

The following is the abstract version of our existence result for Laplacians.

\begin{theorem}\label{T:closable}
Suppose that Assumptions \ref{A:A1}, \ref{A:A2} and \ref{A:A3} hold and $\mu\in \mathcal{M}^{ac}(A)$.
\begin{enumerate}
\item[(i)] The quadratic form $(\mathcal{E},\overline{\mathcal{D}}_0(\mathcal{E}))$ on $L^2(A,\mu)$ is closable, and its closure $(\mathcal{E}, \overline{\mathcal{D}}(\mathcal{E}))$ is a Dirichlet form. We have $\mathbf{1}\in \overline{\mathcal{D}}(\mathcal{E})$, and $(\mathcal{E}, \overline{\mathcal{D}}(\mathcal{E}))$ is local.
\item[(ii)] Its generator $(\overline{\mathcal{L}},\mathcal{D}(\overline{\mathcal{L}}))$ is a non-positive definite self-adjoint operator on $L^2(A,\mu)$. We have $\int_{A} \overline{\mathcal{L}}u\:\varphi\:d\mu=-\mathcal{E}(u,\varphi)$ for all $u\in \mathcal{D}(\overline{\mathcal{L}})$ and $\varphi\in \overline{\mathcal{D}}(\mathcal{E})$. In particular, 
\begin{equation}\label{E:harmonic}
\int_{A} \overline{\mathcal{L}}u\:d\mu=0,\quad u\in \mathcal{D}(\overline{\mathcal{L}}),
\end{equation}
and the bottom of the spectrum of $(-\mathcal{L},\mathcal{D}(\mathcal{L}))$ is zero.
\item[(iii)] Corresponding statements are true for the quadratic form $(\mathcal{E},\mathcal{D}_0(\mathcal{E}))$, its closure $(\mathcal{E}, \mathcal{D}(\mathcal{E}))$ and the generator $(\mathcal{L},\mathcal{D}(\mathcal{L}))$ of the latter. 
\end{enumerate}
\end{theorem}

\begin{remark} 
The Dirichlet form $(\mathcal{E}, \overline{\mathcal{D}}(\mathcal{E}))$ is an extension of $(\mathcal{E}, \mathcal{D}(\mathcal{E}))$.
\end{remark}

The key fact to prove Theorem \ref{T:closable} is the following consequence of (\ref{E:disintegration}).

\begin{lemma}\label{L:closable}
Suppose that Assumptions \ref{A:A1}, \ref{A:A2} and \ref{A:A3} hold and $\mu\in \mathcal{M}^{ac}(A)$. Then $(\mathcal{E}, \overline{\mathcal{D}}_0(\mathcal{E}))$ is closable on $L^2(A,\mu)$.
\end{lemma}

\begin{proof}
Suppose that $(\varphi_n)_{n=1}^\infty\subset \overline{D}_0(\mathcal{E})$ is Cauchy w.r.t. the seminorm $\mathcal{E}^{1/2}$ and such that $\lim_n \left\|\varphi_n\right\|_{L^2(A,\mu)}=0$.  Let $\varepsilon>0$. Choose $n_\varepsilon\geq 1$ large enough so that $\mathcal{E}(\varphi_n-\varphi_m)^{1/2}<\varepsilon/2$ for all $n,m\geq n_\varepsilon$ and choose $\ell\geq 1$ large enough such that 
\[\left(\int_{A\setminus A_\ell}\left\|\nabla \varphi_{n_\varepsilon}(x)\right\|_{T_xW(x)}^2\mu(dx)\right)^{1/2}<\frac{\varepsilon}{2}.\]
Then by the triangle inequality we have 
\begin{equation}\label{E:smallerthaneps}
\sup_{n\geq n_\varepsilon}\left(\int_{A\setminus A_\ell} \left\|\nabla \varphi_{n}(x)\right\|_{T_xW(x)}^2\mu(dx)\right)^{1/2}<\varepsilon.
\end{equation}
We claim that 
\begin{equation}\label{E:claim}
\lim_n \int_{A_\ell}\left\|\nabla \varphi_{n}(x)\right\|_{T_xW(x)}^2\mu(dx)=0.
\end{equation}
If so, then in combination with (\ref{E:smallerthaneps}) we obtain $\lim_n \mathcal{E}(\varphi_n)<\varepsilon$, and since $\varepsilon$ was arbitrary, this shows the closability of $(\mathcal{E}, \overline{\mathcal{D}}_0(\mathcal{E}))$. 

To verify (\ref{E:claim}) let $\mathcal{R}_{\ell,1},...,\mathcal{R}_{\ell,n_\ell}$ be a finite cover of $A_\ell$ by abstract rectangles $\mathcal{R}_{\ell,j}$. For each $j$ the quantity
\begin{multline}\label{E:makesmall}
\int_{\mathcal{P}_{\ell,j}}\int_B \left\|\nabla_{W(\zeta)}(\varphi_n-\varphi_m)(\zeta)\right\|_{T_\zeta W(\zeta)}^2\mu_B(d\zeta)\hat{\mu}_{\mathcal{P}_{\ell,j}}(dB)\\
=\int_{\mathcal{R}_{\ell,j}}\left\|\nabla_{W(x)}(\varphi_n-\varphi_m)(x)\right\|_{T_x W(x)}^2\mu(dx)
\end{multline}
can be made arbitrarily small if $n$ and $m$ are chosen large enough. Here we have used (\ref{E:disintegration}). Clearly also 
\begin{equation}\label{E:gotozero}
\lim_n \int_{\mathcal{P}_{\ell,j}}\int_B (\varphi_n(\zeta))^2\mu_B(d\zeta)\hat{\mu}_{\mathcal{P}_{\ell,j}}(dB)
=\lim_n\int_{\mathcal{R}_{\ell,j}} (\varphi_n(x))^2\mu(dx)=0.
\end{equation}
Each $B\in \mathcal{P}_{\ell,j}$ is an open subset of some Riemannian manifold $W=W(x)$, hence itself a Riemannian manifold. By Proposition \ref{P:closed} the Dirichlet integral 
\begin{equation}\label{E:DirichletonB}
D^{(\mu_B)}(\psi)=\int_B \left\|\nabla_W \psi(\zeta)\right\|_{T_\zeta W}^2\mu_B(d\zeta)
\end{equation}
on $B$ with domain $C^1(B)\cap L^2(B,\mu_B)$ is closable on $L^2(B,\mu_B)$. Therefore also the quadratic form 
\begin{multline}
\varphi\mapsto \int_{\mathcal{R}_{\ell,j}}\left\|\nabla_{W(x)} \varphi (x)\right\|_{T_x W(x)}^2\mu(dx)\notag\\
=\int_{\mathcal{P}_{\ell,j}}\int_B \left\|\nabla_{W(\zeta)}\varphi (\zeta)\right\|_{T_\zeta W(\zeta)}^2\mu_B(d\zeta)\hat{\mu}_{\mathcal{P}_{\ell,j}}(dB),
\end{multline}
endowed with the domain $\overline{\mathcal{D}}_0(\mathcal{E})$, is closable on $L^2(\mathcal{R}_{\ell,j}, \mu)$, see Proposition \ref{P:superpos}. Together with (\ref{E:makesmall}) and (\ref{E:gotozero}) and the triangle inequality for the seminorm $\mathcal{E}^{1/2}$ we have
\begin{align}
\lim_n \left(\int_{A_\ell} \left\|\nabla \varphi_n(x)\right\|_{ T_xW(x)}^2 \mu(dx)\right)^{1/2}&\leq \sum_{j=1}^{n_\ell}\lim_n\left(\int_{\mathcal{R}_{\ell,j}}\left\|\nabla \varphi_n(z)\right\|_{ T_zW(z)}^2 \mu(dz)\right)^{1/2}\notag\\
&=0,\notag
\end{align}
what shows (\ref{E:claim}).
\end{proof}

We prove Theorem \ref{T:closable}.

\begin{proof}
As $(\mathcal{E}, \overline{\mathcal{D}}_0(\mathcal{E}))$ is densely defined by Proposition \ref{P:dense} and closable by Lemma \ref{L:closable}, its closure is a densely defined nonnegative definite symmetric bilinear closed form $(\mathcal{E}, \overline{\mathcal{D}}(\mathcal{E}))$ on $L^2(A,\mu)$. The chain rule for the gradients $\nabla_{W(x)}$ and definition (\ref{E:defquadform}) imply that for any $\varphi\in \overline{\mathcal{D}}_0(\mathcal{E})$ and any $F\in C^1(\mathbb{R})$ with $F(0)=0$ we have $F(\varphi)\in \overline{\mathcal{D}}_0(\mathcal{E})$ and $\mathcal{E}(F(\varphi))\leq \left\|F'\right\|_{\sup}^2\mathcal{E}(\varphi)$. This property extends to all $\varphi\in\overline{\mathcal{D}}(\mathcal{E})$, \cite[Theorem 3.1.1]{FOT94}, and it also implies the Markov property, \cite[Section 1.1]{FOT94}. Since $\mu$ is finite, $\mathbf{1}\in \overline{\mathcal{D}}_0(\mathcal{E})$, and the fact that $\mathcal{E}(\mathbf{1})=0$ follows from (\ref{E:defquadform}). The first statement in (ii) is clear by the general theory mentioned earlier. Since $\mathbf{1}\in \overline{\mathcal{D}}(\mathcal{E})$ and $\mathcal{E}(\mathbf{1})=0$, (\ref{E:harmonic}) follows, and so does the statement on the bottom of the spectrum. The same proof works for (iii).
\end{proof}

Under the hypotheses of Theorem \ref{T:closable} the family $(T_xW(x))_{x\in A}$ may be viewed as a measurable field of Hilbert spaces on $A$, see \cite[Part II, Chapter 1]{Dix}, \cite[Chapter IV, Section 8]{Tak} or \cite[Section 2]{HRT13}, and we write $L^2(A, (T_xW(x))_{x\in A},\mu)$ for its direct integral with respect to $\mu$. By $\Gamma$ we denote the carr\'e du champ operator of $(\mathcal{E},\overline{\mathcal{D}}(\mathcal{E}))$, that is, the unique positive definite continuous symmetric bilinear $L^1(A,\mu)$-valued map on $\overline{\mathcal{D}}(\mathcal{E})$ such that 
$\mathcal{E}(uw,v)+\mathcal{E}(vw,u)-\mathcal{E}(w,uv)=\int_A w\:\Gamma(u,v)\:d\mu$ 
for all $u,v,w\in\overline{\mathcal{D}}(\mathcal{E})\cap L^\infty(A,\mu)$, \cite[Chapter I, Section 4.1]{BH91}. Theorem \ref{T:closable} implies that the leafwise gradient (\ref{E:classgrad}) extends from $\overline{\mathcal{D}}_0(\mathcal{E})$ to the $L^2$-context, that this extension can be used to express the carr\'e du champ, and that an associated divergence operator can be defined as minus its adjoint.

\begin{corollary}\label{C:firstorder} Suppose that Assumptions \ref{A:A1}, \ref{A:A2} and \ref{A:A3} hold and $\mu\in \mathcal{M}^{ac}(A)$.
\begin{enumerate}
\item[(i)] The operator $(\nabla, \overline{\mathcal{D}}_0(\mathcal{E}))$ extends to a densely defined closed unbounded linear operator $\nabla:L^2(A,\mu)\to L^2(A, (T_xW(x))_{x\in A},\mu)$
with domain $\overline{\mathcal{D}}(\mathcal{E})$, and (\ref{E:defquadform}) generalizes to
\[\mathcal{E}(\varphi):=\int_A \left\|\nabla\varphi (x)\right\|_{T_xW(x)}^2\:\mu(dx),\quad \varphi \in \overline{\mathcal{D}}(\mathcal{E}).\]
In particular, the carr\'e du champ operator $\Gamma$ of $\mathcal{E}$ satisfies 
\[\Gamma(\varphi)(x)=\left\|\nabla\varphi (x)\right\|_{T_xW(x)}^2\quad \text{for $\mu$-a.e. $x\in A$}\]
and all $\varphi\in \overline{\mathcal{D}}(\mathcal{E})$.
\item[(ii)] The adjoint $(-\diverg,\mathcal{D}(\diverg))$ of $(\nabla, \overline{\mathcal{D}}(\mathcal{E}))$ is a
densely defined closed unbounded operator $-\diverg:L^2(A, (T_xW(x))_{x\in A},\mu)\to L^2(A,\mu)$.
\item[(iii)] A function $\varphi\in \overline{\mathcal{D}}(\mathcal{E})$ is an element of $\mathcal{D}(\overline{\mathcal{L}})$ if and only if $\nabla \varphi\in \mathcal{D}(\diverg)$, and in this case we have $\overline{\mathcal{L}}\varphi=\diverg(\nabla \varphi)$.
\end{enumerate}
Corresponding statements hold for $(\mathcal{E},\mathcal{D}(\mathcal{E}))$ and the associated objects.
\end{corollary}

\begin{proof}
Statement (i) is immediate from Theorem \ref{T:closable} (i), (ii) follows by general theory, \cite[Theorem VIII.1]{RS80}, (iii) is clear by Theorem \ref{T:closable} (ii).
\end{proof}

Recall that a strongly continuous semigroup $(P_t)_{t>0}$ on $L^2(A,\mu)$ is said to be \emph{symmetric} if $\left\langle P_tu,v\right\rangle_{L^2(A,\mu)}= \left\langle u, P_tv\right\rangle_{L^2(A,\mu)}$
for every $u,v\in L^2(A,\mu)$ and all $t>0$, and \emph{(sub-) Markov} if for all $t>0$ and every $u\in L^2(A,\mu)$ such that $0\leq u\leq 1$ $\mu$-a.e. we have $0\leq P_tu\leq 1$ $\mu$-a.e. See \cite[Sections I.1 and I.2]{BH91} or \cite[Section 1.4]{FOT94}. A symmetric Markov semigroup $(P_t)_{t>0}$ on $L^2(A,\mu)$ is \emph{conservative} if $P_t 1=1$ for all $t>0$
and \emph{recurrent} if for every nonnegative $u\in L^1(A,\mu)$ we have $\int_0^\infty P_tu\:dt=\text{$0$ or $+\infty$}$ $\mu$-a.e. See \cite[p. 48/49]{FOT94}. The following statement is a consequence of \cite[Lemma 1.3.2 and Theorem 1.4.1]{FOT94} (see also \cite[Chapter I, Proposition 3.2.1]{BH91}) and \cite[Theorem 1.6.3 and Lemma 1.6.5]{FOT94}, its last claim is immediate from (\ref{E:harmonic}).

\begin{corollary}\label{C:semigroup} 
Suppose that Assumptions \ref{A:A1}, \ref{A:A2} and \ref{A:A3} hold and $\mu\in \mathcal{M}^{ac}(A)$. There is a unique symmetric Markov semigroup $(\overline{P}_t)_{t>0}$ on $L^2(A,\mu)$ generated by $(\overline{\mathcal{L}},\mathcal{D}(\overline{\mathcal{L}}))$ as in Theorem \ref{T:closable} (ii). It is recurrent and conservative, and $\mu$ is an invariant measure for $(\overline{P}_t)_{t>0}$ in the sense that $\int_{A} \overline{P}_t u\: d\mu=\int_{A} u \:d\mu$ for all $u\in L^2(A,\mu)$. Corresponding statements are true for the unique symmetric Markov semigroup $(P_t)_{t>0}$ on $L^2(A,\mu)$ generated by $(\mathcal{L},\mathcal{D}(\mathcal{L}))$ as in Theorem \ref{T:closable} (iii). 
\end{corollary}

If the support of $\mu$ is all of $A$ then the theory in \cite{FOT94} applies. In this case a Dirichlet form $(\mathcal{E},\mathcal{D}(\mathcal{E}))$ on $L^2(A,\mu)$ is called \emph{regular} if $\mathcal{D}(\mathcal{E})\cap C_c(A)$ is dense in $\mathcal{D}(\mathcal{E})$ and uniformly dense in $C_c(A)$. If it is regular it is said to be \emph{strongly local} if $\mathcal{E}(u,v)=0$ whenever $u,v\in \mathcal{D}(\mathcal{E})\cap C_c(A)$ are such that $v$ is constant on $\supp u$. If a Dirichlet form is regular and strongly local, then there is an associated symmetric diffusion process. Since it involves some technical notions, we first state the following theorem and then comment on these notions in Remark \ref{R:notions} below. By $d|_{A\times A}$ we denote the restriction of the geodesic distance $d$ on $M$ to $A$. 
 
\begin{theorem}\label{T:regular} Suppose that Assumptions \ref{A:A1}, \ref{A:A2} and \ref{A:A3} hold, $\mu\in \mathcal{M}^{ac}(A)$ and $\supp\mu=A$. Then $(\mathcal{E},\mathcal{D}(\mathcal{E}))$ is a strongly local regular Dirichlet form on $(A, d|_{A\times A})$. There is a symmetric Hunt diffusion process $((X_t)_{t\geq 0}, (\mathbb{P}^x)_{x\in A\setminus \mathcal{N}})$ with starting point $x$ outside some properly exceptional Borel set $\mathcal{N}\subset A$, such that for any bounded Borel function $u$ on $A$ and any $t>0$ the function $\mathcal{P}_tu$, defined by 
\begin{equation}\label{E:qcversions}
\mathcal{P}_tu(x):=\begin{cases} \mathbb{E}^x[u(X_t)]& \text{for $x\in A\setminus \mathcal{N}$} \\ 0& \text{for $x\in \mathcal{N}$},\end{cases}
\end{equation}
satisfies
\begin{equation}\label{E:coincideae}
\mathcal{P}_tu(x)=P_tu(x)\quad \text{for $\mu$-a.e. $x\in A$.}
\end{equation}
This process is unique up to equivalence. It has infinite life time. Moreover, for each bounded Borel $u$ and $t>0$ 
the function $\mathcal{P}_tu$ is an $(\mathcal{E},\mathcal{D}(\mathcal{E}))$-quasi-continuous version of $P_tu$.
\end{theorem}
 
\begin{remark}\label{R:notions}
A \emph{diffusion process} is a strong Markov process with almost surely continuous paths, \cite[Section 4.5]{FOT94}. A \emph{Hunt process} is a particularly regular type of strong Markov process, see \cite[Section I.9]{BG68} or \cite[Appendix A.2]{FOT94}. A Borel set $\mathcal{N}\subset A$ is called \emph{properly exceptional} if it is a $\mu$-null set and $\mathbb{P}^x(X_t\in\mathcal{N}\ \text{for some $t\geq 0$})=0$, $x\in A\setminus \mathcal{N}$, see \cite[p. 134 and Theorem 4.1.1]{FOT94}. The process $((X_t)_{t\geq 0}, (\mathbb{P}^x)_{x\in A\setminus \mathcal{N}})$ is said to be \emph{symmetric} (with respect to $\mu$) if for any $t>0$ and any bounded Borel functions $u,v$ on $A$ we have $\left\langle \mathcal{P}_tu,v\right\rangle_{L^2(A,\mu)}= \left\langle u, \mathcal{P}_tv\right\rangle_{L^2(A,\mu)}$, \cite[Lemma 4.1.3]{FOT94}. Two symmetric Hunt processes are said to be \emph{equivalent} if they have a common properly exceptional set $\mathcal{N}$ such that $\mathcal{P}^{(1)}_tu(x)=\mathcal{P}^{(2)}_tu(x)$, $x\in A\setminus \mathcal{N}$, for all $t>0$ and bounded Borel $u$, see \cite[p. 147]{FOT94}; here $\mathcal{P}^{(1)}_tu$ and $\mathcal{P}^{(2)}_tu$ are defined for the respective process as in the theorem. Quasi-continuous versions in the context of regular Dirichlet forms are discussed in \cite[Section 2.1]{FOT94}.
\end{remark} 
 
\begin{proof}
The statements on regularity and strong locality are immediate. The existence of a symmetric Hunt diffusion process is guaranteed by \cite[Theorems 7.2.1 and 7.2.2]{FOT94}, its uniqueness up to equivalence by \cite[Theorem 4.2.7]{FOT94}. The quasi-continuity of the functions $\mathcal{P}_tu$ and (\ref{E:coincideae}) are verified in \cite[Theorem 4.2.3]{FOT94}.
That $((X_t)_{t\geq 0}, (\mathbb{P}^x)_{ x\in A\setminus \mathcal{N}})$ has infinite life time follows from the conservativeness of $(\mathcal{E},\mathcal{D}(\mathcal{E}))$, \cite[Problem 4.5.1]{FOT94}.
\end{proof}

\section{Hyperbolic attractors}\label{S:hyperbolic}

Let $M$ be a compact smooth Riemannian manifold and let $f:M\rightarrow M$ be a $C^{r+\alpha}$ diffeomorphism, $r\geq1$ and $\alpha\in(0,1]$. 

A compact subset $\Lambda\subset M$ is called a \textit{(topological) attractor} for $f$ if there is an open neighborhood $U$ of $\Lambda$ such that $\overline{f(U)}\subset U$ and
\begin{equation}
\Lambda=\underset{n\geq 0}\bigcap\, f^n(U).
\end{equation}
Clearly $\Lambda$ is the largest $f$-invariant subset of $U$, $f(\Lambda)=\Lambda$. If $\Lambda$ is an  attractor, then $f|_\Lambda$ is called  \textit{topologically transitive} if for any nonempty open sets $V,W\subset \Lambda$ there is some $n\in\mathbb{N}$ such that $f^n(V)\cap W\neq \emptyset$. See for instance \cite{BrinStuck, KatokHasselblatt}.

A compact $f$-invariant subset $\Lambda\subset M$ is said to be \textit{partially hyperbolic} if for each $x\in \Lambda$ there exists a splitting of the tangent space 
\[T_xM=E^s(x)\oplus E^c(x)\oplus E^u(x)\] 
and there are constants $c>0$, $0<\lambda_1\leq\mu_1<\lambda_2\leq\mu_2<\lambda_3\leq\mu_3$ with $\mu_1<1<\lambda_3$ such that for each $x\in \Lambda$ we have 

\begin{equation*}
d_xf E^s(x)=E^s(f(x)),\quad d_xf E^c(x)=E^c(f(x)),\quad d_xf E^u(x)=E^u(f(x)),
\end{equation*}
and
\begin{align}
 c^{-1}\lambda^n_1\left\| v\right\|_{T_{x}M}\leq\left\|d_xf^n v\right\|_{T_{f^n(x)}M}&\leq c\mu^n_1 \left\| v \right\|_{T_xM} \qquad \text{for $v\in E^s(x)$ and $n\geq 0$},\notag\\
 c^{-1}\lambda^n_2\left\| v\right\|_{T_{x}M}\leq\left\|d_xf^n v\right\|_{T_{f^n(x)}M}&\leq c\mu^n_2 \left\| v \right\|_{T_xM} \qquad \text{for $v\in E^c(x)$ and $n\geq 0$},\notag\\
 c^{-1}\lambda^n_3\left\| v\right\|_{T_{x}M}\leq\left\|d_xf^n v\right\|_{T_{f^n(x)}M}&\leq c\mu^n_3 \left\| v \right\|_{T_xM} \qquad \text{for $v\in E^u(x)$ and $n\geq 0$}.\notag
\end{align}
The subspaces $E^s(x)$, $E^u(x)$ and $E^c(x)$ (called the \emph{stable}, \emph{unstable} and \emph{central} subspaces at $x$, respectively) depend H\"older continuously on $x\in\Lambda$, see \cite[Theorem 4.11 and Exercise 4.15]{BarreiraPesin} or 
\cite[Theorem 2.3]{Pesin1}. The map $f$ is called a \emph{partially hyperbolic diffeomorphism} if $M$ is a partially hyperbolic set. If $E^c(x)=\lbrace0\rbrace$ then $\Lambda$ is said to be \emph{uniformly hyperbolic}. If $M$ itself is a uniformly hyperbolic set then $f$ is called an \emph{Anosov diffeomorphism}. A topological attractor $\Lambda\subset M$ is called a \emph{partially} (resp. \emph{uniformly}) \emph{hyperbolic attractor} if it is a partially (resp. uniformly) hyperbolic set.

\begin{remark}
The tangent vectors in $E^c(x)$ may be contracted or expanded, but not as sharply as vectors in $E^s(x)$ and $E^u(x)$; the central direction is `dominated' by the hyperbolic behavior of the stable and unstable directions. See \cite[Section 2.1.4]{HaPe} or \cite[p. 13/14]{Pesin1}.
\end{remark}

One of the central results in the theory of hyperbolic dynamical systems is the
Stable Manifold Theorem, see \cite[Sections 4.2--4.5]{Pesin1}. We quote it to fix notation and as a reference.

\begin{theorem}\label{T:SMT2} Let $f:M\rightarrow M$ be a $C^{r+\alpha}$ diffeomorphism of a compact smooth Riemannian manifold $M$, and $\Lambda\subset M$ be a partially hyperbolic set. Then, for every $x\in \Lambda$ there are $C^r$ embedded submanifolds $V^s(x)$ and $V^u(x)$ such that
\begin{enumerate}
\item[(i)] We have $T_xV^{s}(x)=E^s(x)$ and $T_xV^{u}(x)=E^u(x)$.
\item[(ii)] We have $f(V^s(x))\subset V^s(f(x))$ and $f^{-1}(V^u(x))\subset V^u(f^{-1}(x))$.
\item[(iii)] For any $y\in V^{u}(x)$ and any $n$ we have 
\begin{equation}\label{E:expansion}
d(f^{-n}(x),f^{-n}(y))\leq C\lambda^n d(x,y)
\end{equation} 
with constants $\lambda\in(0,1)$ and $C>0$ independent of $x$, $y$ and $n$.
\end{enumerate}
\end{theorem}

The embedded submanifolds $V^s(x)$ and $V^u(x)$
are called \emph{local stable} and \emph{local unstable manifolds}. Given a point $x\in \Lambda$, the \textit{global stable} and \textit{global unstable manifolds} of $x$ are defined as
\begin{equation}\label{E:localtoglobal}
W^{s}(x)=\underset{n\geq0}\bigcup f^{-n}(V^{s}(f^{n}(x))) \quad \text{ and }\quad 
W^{u}(x)=\underset{n\geq0}\bigcup f^{n}(V^{u}(f^{-n}(x))).
\end{equation}

\begin{remark}\label{R:strongly}
In the literature on partially hyperbolic systems the manifolds $V^s$ (resp. $V^u$) and $W^s$ (resp. $W^u$) as defined here are also referred to as \emph{strongly} local stable (resp. unstable) and \emph{strongly} global stable (resp. unstable) manifolds.
\end{remark}

Being increasing unions of embedded submanifolds, $W^{s}(x)$ and $W^{u}(x)$ are $C^r$ injectively immersed submanifolds of $M$,\cite[Chapter VII, Proposition 1.3]{Robinson}. Any partially hyperbolic attractor $\Lambda$ contains the global unstable manifolds of its points, 
\begin{equation}\label{E:partiallyglobal}
\Lambda=\bigcup_{x\in\Lambda} W^u(x), 
\end{equation}
see for instance \cite[Theorem 9.1]{HaPe}, and the union (\ref{E:partiallyglobal}) is disjoint. The set $\mathcal{F}^u=\{W^u(x): x\in \Lambda\}$ forms a continuous lamination on $\Lambda$ in the sense of \cite[Definition A.3.5]{KatokHasselblatt}, and the $W^u(x)$ are also referred to as \emph{(global, unstable) leaves}. 

\begin{remark}\label{R:Anosov}
If $f$ is an Anosov diffeomorphism the set $\mathcal{F}^u$ forms a H\"older continuous foliation on $M$ with $C^r$ leaves in the sense of \cite[Section 5.13]{BrinStuck}, see also \cite[Section 1.1, p.7]{BarreiraPesin} or \cite[remarks after Definition A.3.5]{KatokHasselblatt}. In general this foliation is not even Lipschitz, \cite{A67}, and we cannot claim that $(M, \mathcal{F}^u)$ is a foliated manifold in the sense of \cite[Definition 1.1.18]{CandelConlonI}.
\end{remark}

\begin{remark}\label{R:expanding}
If $\Lambda$ is a uniformly hyperbolic attractor and its topological dimension equals the dimension of the unstable leaves, $\dim \Lambda=\dim V^{u}(x)$, then $\Lambda$ is called an \emph{expanding attractor}, \cite[Definition (4.2)]{Williams}. An 
$n$-solenoid $\Sigma$ in the sense of \cite[Definition (3.1)]{Williams} is a foliated space in the sense of \cite[Definition 11.2.12]{CandelConlonI}, and one can find a manifold $M$ and a $C^r$ diffeomorphism $f:M\to M$ having an expanding attractor $\Lambda$ such that $f|_\Lambda$ is conjugate to the shift map $h:\Sigma\to \Sigma$ of $\Sigma$. This is proved in \cite[Theorems B and C]{Williams}, see also \cite[p. 284]{CandelConlonI}. Under additional regularity assumptions one can also find a conjugate $n$-solenoid for a given expanding attractor, \cite[Theorem A]{Williams}. 
\end{remark}

By $B(x,\delta)$ we denote the open ball in $M$ with center $x\in M$ and radius $\delta>0$.
Let $\Lambda$ be a partially hyperbolic attractor. Given $x\in \Lambda$ and $\delta>0$ the set
\begin{equation}\label{E:rectanglePH}
\mathcal{R}=\mathcal{R}(x,\delta)=\bigcup_{z\in B(x,\delta)\cap \Lambda}V^u(z),
\end{equation}
is called a \emph{rectangle} at $x$, see \cite{PesinSinai}. It follows from the definition that for any point $x$ in $\Lambda$ we can find a rectangle $\mathcal{R}_x$ and an open neighborhood $U_x\subset M$ of $x$ such that $U_x\cap\Lambda\subset \mathcal{R}_x$. Moreover, since $\Lambda$ is a compact subset of $M$, we can find $x_1,...,x_N\in \Lambda$ so that
\begin{equation}\label{E:coverLambda}
\Lambda\subset \bigcup_{i=1}^N (U_{x_i}\cap\Lambda)\subset\bigcup_{i=1}^N \mathcal{R}_{x_i}.
\end{equation}

Let $\mu$ be a Borel probability measure on $\Lambda$ and $\mathcal{R}$ a rectangle with $\mu(\mathcal{R})>0$ partitioned into local unstable manifolds as in (\ref{E:rectanglePH}). Since this partition, denoted by $\mathcal{P}_\mathcal{R}$, is measurable, cf. Appendix \ref{S:Rokhlin}, the disintegration identity
\begin{equation}\label{E:disintegration1}
\mu(E)=\int_{\mathcal{P}_\mathcal{R}} \, \mu_{V^u}(E)\, \hat{\mu}_{\mathcal{P}_\mathcal{R}}(dV^u), \quad \text{$E\subset \mathcal{R}$ Borel,}
\end{equation}
holds, where $\hat{\mu}_{\mathcal{P}_\mathcal{R}}$ is the pushforward of $\mu$ under the canonical projection onto the elements $V^u$ of the partition $\mathcal{P}_\mathcal{R}$ and for each $V^u\in \mathcal{P}_\mathcal{R}$ the symbol $\mu_{V^u}$ denotes the conditional measure on $V^u$. By $m_{V^u}$ we denote the Riemannian volume induced on the local unstable manifold $V^u$ (the \emph{leaf volume} on $V^u$).

An $f$-invariant Borel probability measure $\mu$ on a partially hyperbolic attractor $\Lambda$ is called a \textit{Gibbs $u$-measure} if for any rectangle $\mathcal{R}$ of positive measure the conditional measures $\mu_{V^u}$ in (\ref{E:disintegration1}) are absolutely continuous with respect to $m_{V^u}$ for $\hat{\mu}_{\mathcal{P}_\mathcal{R}}$-a.e. $V^u$. A Gibbs $u$-measure which is, in addition, hyperbolic in the sense of \cite[Definition 5.4.5 and p. 418]{BarreiraPesin2} is known as an \emph{SRB measure}, see also \cite {CLP17}. In particular, a Gibbs $u$-measure on a uniformly hyperbolic attractor is SRB.

We quote the following result from \cite[Theorem 4]{PesinSinai}; further details and descriptions can be found in \cite[Section 5]{CLP17}, \cite[Section 9]{HaPe}; see also \cite[Theorem 5.4]{CLP17}, \cite{BDPP08} or \cite[Theorems 9.3.4 and 9.3.6]{BarreiraPesin2}. A brief sketch of the geometric construction of Gibbs $u$-measures is provided in Appendix \ref{S:SRB}.

\begin{theorem}\label{T:Gibbsexist}
Let $\Lambda\subset M$ be a partially hyperbolic attractor associated with a $C^{1+\alpha}$ diffeomorphism
$f:M\to M$.
\begin{enumerate}
\item[(i)] There is a Gibbs $u$-measure $\mu$ on $\Lambda$.
\item[(ii)] If $\mu$ is a Gibbs $u$-measure on $\Lambda$, then the conditional densities $d\mu_{V^u}/dm_{V^u}$ of $\mu$ are H\"older continuous, bounded and uniformly bounded away from zero.
\item[(iii)] If for every $x\in\Lambda$ the orbit of the global (strongly) unstable manifold $W^u(x)$ is dense in $\Lambda$, then every Gibbs $u$-measure $\mu$ has support $\supp\mu=\Lambda$.
\item[(iv)] If $\Lambda$ is a uniformly hyperbolic attractor and $f|_\Lambda$ is topologically transitive, then the measure $\mu$ as in (i) is an SRB measure with support $\supp\mu=\Lambda$. Moreover, it is the only SRB measure on $\Lambda$ and it is 
ergodic.
\end{enumerate}
\end{theorem}

\begin{remark}
The uniqueness of Gibbs $u$-measure can be guaranteed under some additional assumptions, see for instance \cite[Section 5.4, Theorem 5.5]{CLP17}. 
\end{remark}

\begin{remark} 
Recall that a measure on $M$ is called \emph{smooth} if it has a continuous density with respect to the Riemannian volume on $M$. For Anosov diffeomorphisms the existence of an ergodic SRB measure is classical, \cite{Anosov67}, see \cite[Theorem 1.1]{BarreiraPesin}: Any $C^{1+\alpha}$ Anosov diffeomorphism of a smooth compact Riemannian manifold $M$ keeping a smooth measure invariant is ergodic with respect to this measure, and this measure is an SRB measure on $M$. 
\end{remark}

\begin{theorem}\label{T:pha}
Let $\Lambda\subset M$ be a partially hyperbolic attractor associated with a $C^{1+\alpha}$ diffeomorphism
$f:M\to M$ and let $\mu$ be a Gibbs $u$-measure on $\Lambda$.
\begin{enumerate}
\item[(i)] Assumptions \ref{A:A1}, \ref{A:A2} and \ref{A:A3} are satisfied for $A=\Lambda$, and $\mu\in \mathcal{M}^{ac}(\Lambda)$. 
\item[(ii)] Let $\overline{\mathcal{D}}_0(\mathcal{E})$, $\mathcal{E}$ and $\mathcal{D}_0(\mathcal{E})$ be defined as in (\ref{E:initialdomain}), (\ref{E:defquadform}) and (\ref{E:secdom}) with $A=\Lambda$. The forms $(\mathcal{E},\overline{D}_0(\mathcal{E}))$ and $(\mathcal{E},\mathcal{D}_0(\mathcal{E}))$ are closable on $L^2(\Lambda,\mu)$; their closures $(\mathcal{E},\overline{\mathcal{D}}(\mathcal{E}))$ and $(\mathcal{E},\mathcal{D}(\mathcal{E}))$ are local Dirichlet forms. Moreover, $\mathbf{1}\in \mathcal{D}(\mathcal{E})\subset \overline{\mathcal{D}}(\mathcal{E})$. Their generators $(\overline{\mathcal{L}},\mathcal{D}(\overline{\mathcal{L}}))$ respectively $(\mathcal{L},\mathcal{D}(\mathcal{L}))$ are non-positive definite self-adjoint operators, and they satisfy (\ref{E:harmonic}) with $A=\Lambda$.
\item[(iii)] The unique symmetric Markov semigroups $(\overline{P}_t)_{t>0}$ and $(P_t)_{t>0}$ on $L^2(\Lambda,\mu)$ generated by $(\overline{\mathcal{L}},\mathcal{D}(\overline{\mathcal{L}}))$ respectively $(\mathcal{L},\mathcal{D}(\mathcal{L}))$ are recurrent and conservative, and they leave $\mu$ invariant. 
\item[(iv)] If in addition the conditions of Theorem \ref{T:Gibbsexist} (ii) or (iii) are met, then $(\mathcal{E},\mathcal{D}(\mathcal{E}))$ is a strongly local regular Dirichlet form and there is a symmetric Hunt diffusion process $((X_t)_{t\geq 0}, (\mathbb{P}^x)_{x\in \Lambda\setminus \mathcal{N}})$ on $\Lambda$ associated with $(\mathcal{E},\mathcal{D}(\mathcal{E}))$ as in (\ref{E:coincideae}), and it is unique up to equivalence.
\end{enumerate}
\end{theorem}

The operators $(\overline{\mathcal{L}},\mathcal{D}(\overline{\mathcal{L}}))$ respectively $(\mathcal{L},\mathcal{D}(\mathcal{L}))$ may be viewed as a \emph{natural self-adjoint Laplacian on $\Lambda$} and the diffusion $((X_t)_{t\geq 0}, (\mathbb{P}^x)_{x\in \Lambda\setminus \mathcal{N}})$ as a \emph{natural symmetric leafwise diffusion}.

\begin{proof}
By Theorem \ref{T:SMT2}, (\ref{E:localtoglobal}) and (\ref{E:partiallyglobal}) Assumption \ref{A:A1} is satisfied. The H\"older continuity of the dependence $x\mapsto E^u(x)$ implies Assumption \ref{A:A2}. By (\ref{E:rectanglePH}) and (\ref{E:coverLambda})  Assumption \ref{A:A3} holds with $A_\ell:=\Lambda$, $\ell\geq 1$. Since $\mu\in \mathcal{M}^{ac}(\Lambda)$, the results now follow from Theorems \ref{T:closable} and \ref{T:regular} and Corollary \ref{C:semigroup}.
\end{proof}

\begin{remark} If $r=1$ then the $V^u(x)$ are of class $C^1$ only, so that no classical leafwise Laplacian can be defined, but self-adjoint Laplacians exist by Theorem \ref{T:pha}. Even if $r\geq 2$ it is not clear how to define related classical leafwise Laplacians if, as in Theorem \ref{T:Gibbsexist}, the densities are only known to be H\"older.
\end{remark}

\begin{remark}
The Laplacians on foliated spaces in \cite[Sections 2.1 and 2.2]{CandelConlonII} are constructed from classical Laplacians on global leaves $W$ of form $\diverg_{m_{W}}\circ \nabla_{W}$; here $\diverg_{m_{W}}$ denotes the divergence 
operator on $W$ with respect to the Riemannian volume $m_W$. They are Feller generators, and they were used to obtain structurally interesting volume measures, \cite{Candel03, CandelConlonI, CandelConlonII, 
Garnett83}. In general this approach does not lead to self-adjoint Laplacians. In our situation the volume measure is already given by a Gibbs $u$-measure, and the Laplacians in Theorem \ref{T:closable} are locally of form $\diverg_{\mu_{V^u(z)}}\circ \nabla_{V^u(z)}$ with divergence operators $\diverg_{\mu_{V^u(z)}}$ defined with respect to the weighted measures $\mu_{V^u(z)}$ on the local unstable manifolds. Here self-adjointness is clear by construction.
\end{remark}

\begin{remark}\label{R:index}
In the situation of Theorem \ref{T:pha} (iv) the \emph{index} of $(\mathcal{E},\mathcal{D}(\mathcal{E}))$ in the sense of \cite{Hino} is well-defined, and it is easily seen to equal $\dim E^u(x)$. 
\end{remark}

We briefly mention some concrete examples of hyperbolic attractors.

\begin{examples}\label{Ex:classex}\mbox{}
\begin{enumerate}
\item[(i)] Prototype examples of Anosov diffeomorphisms
are hyperbolic toral automorphisms $f_A:\mathbb{T}^n\rightarrow \mathbb{T}^n$ of the $n$-torus $\mathbb{T}^n=\mathbb{R}^n/\mathbb{Z}^n$, defined by $f_A(x)= Ax\mod 1$, where $A$ is an integer $(n\times n)$-matrix with $\vert \det A\vert= 1$ and no eigenvalue of modulus $1$, \cite[Section 1.7]{BrinStuck}, \cite[Section 7.5]{Robinson}. Its direct product with the identity map is then a partially hyperbolic diffeomorphism, \cite[Section 2.3]{Pesin1}.
\item[(ii)] Prototype examples of expanding attractors \cite{Williams}, are the \emph{Smale-Williams solenoids}, \cite[Section 1.9]{BrinStuck},  \cite[Section 17.1]{KatokHasselblatt}, \cite[Section 8.7]{Robinson}. Let $D^2\subset\mathbb{R}^2$ be the unit disk and $S^1$ be the unit circle parametrized by the angle $\theta\in \left[ 0, 2\pi \right)$. 
Given $r\in \left( 0,1\right)$ and $\alpha, \beta \in \left( 0, \min \left\lbrace r, 1-r\right\rbrace \right) $, we can define a smooth map $f:D^2\times S^1\to D^2\times S^1$ by 
$f(x,y,\theta)=(\alpha x+r\cos \theta, \beta y+r\sin\theta, 2\theta)$, and the resulting attractor $\Lambda=\bigcap_{n\geq0}f^n(D^2\times S^1)$ is a Smale-Williams solenoid. Physical interpretation can be found in \cite{Kuznetsov}.
\item[(iii)] The so-called \emph{DA-attractor} and in particular, the \emph{Plykin attractor} \cite{Plykin}, Figure 1, are the other well-known uniformly hyperbolic attractors. See \cite[Sections 17.2a. and 17.2b.]{KatokHasselblatt}, \cite[Sections 8.8 and 8.9]{Robinson}, and also \cite{Coudene},  \cite[Sections 1.2.4, 2.4 and 2.5]{Kuznetsov} for visualizations.

\begin{figure}
\centering
\captionsetup[subfigure]{labelformat=empty}
\begin{subfigure}{.4\textwidth}
\centering
\includegraphics[height=3cm]{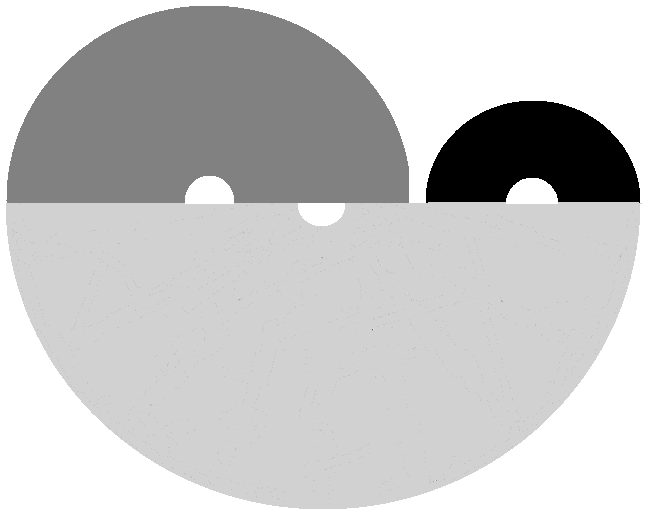}
\caption{$D$}
\end{subfigure}%
{\LARGE$\xrightarrow{f}$}
\begin{subfigure}{.4\textwidth}
\centering
\includegraphics[height=3cm]{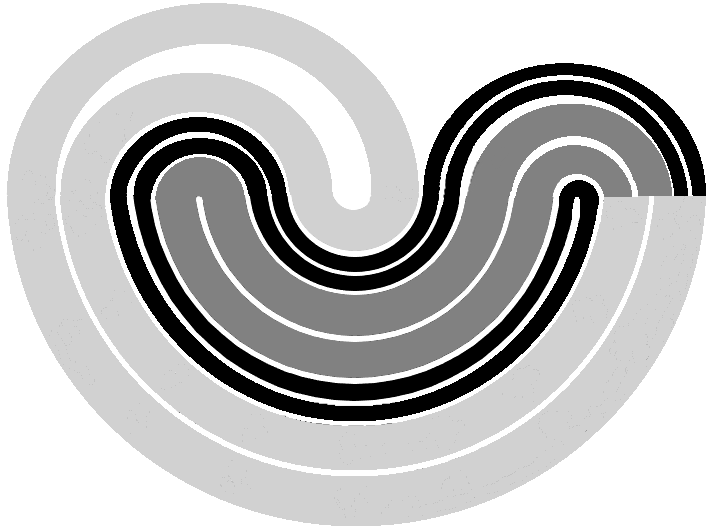}
\caption{$f(D)$}
\end{subfigure}%
\caption{The Plykin attractor in the plane}
\end{figure}
\end{enumerate}
\end{examples}

\section{Geodesic flows on manifolds with negative curvature}\label{S:geodesicflow}

A particular class of examples is formed by geodesic flows on manifolds of negative sectional curvature. Since for these examples there is an established leafwise analysis to which the Laplacians in Theorem \ref{T:pha} can be compared, we discuss them in a slightly more detailed manner.

Let $\phi:\mathbb{R}\times M\rightarrow M$ be a flow (generated by a given vector field). By \textit{time-$t$ map} we mean the diffeomorphism $\phi(t,\cdot) : M \rightarrow M$. The flow $\phi$ is said to be \emph{partially hyperbolic} if its time-$1$ map $\phi(1,\cdot)$ is a partially hyperbolic diffeomorphism. A \emph{uniformly hyperbolic} (or \emph{Anosov}) \emph{flow} is a partially hyperbolic flow with one-dimensional central subspace $E^c(x)=\text{span} \lbrace{\frac{\partial}{\partial t}\vert}_{t=0}\, \phi(t,x)\rbrace$, $x\in M$. See for instance \cite[Definition 17.4.2]{KatokHasselblatt}. If $\phi$ is a $C^{r+\alpha}$ Anosov flow, then the local stable and unstable manifolds can be obtained from a continuous time-version of the Stable Manifold Theorem for flows \cite[Theorem 17.4.3]{KatokHasselblatt}, \cite[Section 7.3.5]{BarreiraPesin}, and also in (\ref{E:localtoglobal}) a positive real index can be used in place of $n$ and $\phi(t,\cdot)$ can replace $f^n$.

Let $M$ be a compact $C^r$ manifold endowed with a $C^r$ Riemannian metric, $r\geq2$. Given $x\in M$ and $v\in T_xM$, there is a unique geodesic $\gamma_{x,v}$ such that $\gamma_{x,v}(0)=x$ and $\dot{\gamma}_{x,v}(0)=v$. By the \emph{geodesic flow} on $M$ we mean the flow $g:\mathbb{R}\times TM\rightarrow TM$ on the tangent bundle $TM$, defined by
\begin{equation}\label{E:geoflow}
g(t,(x,v)):=(\gamma_{x,v}(t),\dot{\gamma}_{x,v}(t)).
\end{equation}
Since the length of the tangent vectors are preserved by the geodesic flow, i.e. 
\[\Vert \dot{\gamma}_{x,v}(t)\Vert_{T_{\gamma_{x,v}(t)}M}=\Vert v \Vert_{T_xM},\]
it is usual to consider the restriction of $g$ to the unit tangent bundle
\[T^1M=\lbrace (x,w)\in TM:\, \Vert w\Vert_{T_xM}=1\rbrace.\]
We assume that $M$ has negative sectional curvature. Then the $C^{r-1}$ geodesic flow $g:\mathbb{R}\times T^1M\rightarrow T^1M$ is an Anosov flow \cite{Anosov67}, see for instance \cite[Theorem 9.4.1]{BG05} and \cite[Sections 17.5 and 17.6]{KatokHasselblatt}. Since the time-$1$ map $g(1,\cdot):T^1M\to T^1M$ of the geodesic flow (\ref{E:geoflow}) is a partially hyperbolic diffeomorphism, Theorem \ref{T:Gibbsexist} ensures the existence of a Gibbs $u$-measure $\mu$ on $T^1M$. On the other hand the \emph{Liouville measure} $\nu$ on $T^1M$, defined by
\begin{equation}\label{E:Liouville}
\int_{T^1M}h\,d\nu=\int_M\int_{T^1_xM}h(x,v)\,dm_{\mathbb{S}^{n-1}}\,dm_M(x),\quad h\in C(T^1M),
\end{equation}
where $m_M$ and $m_{\mathbb{S}^{n-1}}$ denote the Riemannian volume on $M$ and on the sphere $\mathbb{S}^{n-1}\equiv T^1_xM$, respectively, is invariant under $g(1,\cdot)$, see for instance \cite[Appendix C]{Poll93}. But this implies that $\mu=\nu$, \cite[Theorem 7.4.14]{FH19}, and using (\ref{E:Liouville}) it follows that $\supp\mu=T^1M$. It is well known that in the special case of a compact surface of constant negative curvature, this measure also coincides with the Bowen-Margulis measure, i.e. the measure of maximal entropy of the flow, see \cite{L90}.

Now Theorem \ref{T:pha} yields self-adjoint Laplacians  $(\overline{\mathcal{L}},\mathcal{D}(\overline{\mathcal{L}}))$ on $L^2(T^1M,\mu)$. On the other hand, there is an established analysis involving leafwise Laplacians in the unstable directions, see for instance \cite{Yue, Yue95}. We briefly compare $(\overline{\mathcal{L}},\mathcal{D}(\overline{\mathcal{L}}))$ to the Laplacians studied in \cite{Yue, Yue95}.

For any rectangle $\mathcal{R}$ and any local unstable manifold $V^u$ as in (\ref{E:rectanglePH}) let $\varrho_{V^u}=d\mu_{V^u}/dm_{V^u}$. Assume that $M$ and $f$ are of class $C^\infty$. Then each $\varrho_{V^u}$ is in $C^\infty(V^u)$, see \cite[Lemma 2.1]{Yue} and \cite[Lemma 2.5]{LMM86}. Setting $\varrho(z):=\varrho_{V^u}(z)$ if $z$ is a point in the partition element $V^u$, we can define $\varrho$ as a $C^{\infty,u}$-function on $\mathcal{R}$. On the other hand, a ``classical" leafwise Laplacian in the unstable directions can be defined by
\[\Delta \varphi(z):=\Delta_{W^u(z)}(\varphi|_{W^u(z)})(z),\quad z\in M,\quad \varphi\in C^{2,u}(M);\]
here $\Delta_{W^u(z)}$ denotes the classical Laplace-Beltrami operator on $W^u(z)$, defined on $C^2$ functions. Now recall that $\nabla$ defines the leafwise gradient on $C^{1,u}(M)$-functions as in (\ref{E:classgrad}) (with $A=M$ and $W(z)=W^u(z)$). For $\varphi\in C^{2,u}(M)$ we can define 
\begin{equation}\label{E:Laplacian}
\Delta_\mu\varphi:=\Delta\varphi+\varrho^{-1}\langle\nabla \varrho,\nabla \varphi\rangle,\quad z\in \mathcal{R},
\end{equation}
where we write $\langle\nabla \varrho,\nabla \varphi\rangle$ to denote the function $z\mapsto\langle\nabla \varrho (z),\nabla \varphi (z)\rangle_{T_zW^u(z)}$.  Definition (\ref{E:Laplacian}) can meaningfully be extended to hold for all $z\in M$. The Laplacian $\Delta_\mu\varphi$ may be seen as a leafwise analog of the Laplacian on weighted manifolds, \cite[Definition 3.17]{Grigoryanbook}. It had been studied in \cite{Yue}, see for instance \cite[Theorem 1']{Yue}, where $\mu$ had been shown to be an invariant measure for $\Delta_\mu$. This is fully consistent with our results in the sense that 
\begin{equation}\label{E:saext}
\text{$(\overline{\mathcal{L}},\mathcal{D}(\overline{\mathcal{L}}))$ is a self-adjoint extension of $(\Delta_\mu, C^{2,u}(M))$ on $L^2(T^1M,\mu)$.}
\end{equation}
Note that we had independently observed the invariance of $\mu$ in (\ref{E:harmonic}). Observation (\ref{E:saext}) means that for geodesic flows of class $C^\infty$ the theory in this article is an $L^2$-version of the smooth theory for the operator $\Delta_\mu$ as considered in \cite{Yue, Yue95}. A similar $L^2$-theory for the stable directions, including self-adjoint Laplacians, is studied extensively in \cite{Hamenstaedt97}.

\begin{remark}\label{R:orientation2}
For a definition of leafwise Laplacians as in (\ref{E:Laplacian}) it is sufficient to have unstable manifolds of class $C^2$ and densities that are $C^1$. This can also be guaranteed for certain other classes of diffeomorphisms $f:M\to M$, see for instance \cite[Remark on p. 534]{LedrappierYoung1} or \cite[p. 168]{Yue95}.
\end{remark}

\section{Quasi-invariance in the $u$-conformal case}\label{S:quasiinv}

In this section we observe that if the diffeomorphism $f$ is $u$-conformal, then the Dirichlet forms $(\mathcal{E}, \overline{\mathcal{D}}(\mathcal{E}))$ and $(\mathcal{E}, \mathcal{D}(\mathcal{E}))$ are quasi-invariant under the iterates of $f$.

Let $\Lambda$ be a partially hyperbolic attractor of a $C^{1+\alpha}$ diffeomorphism $f:M\rightarrow M$. We additionally assume that $f$ is \emph{$u$-conformal}, \cite[p. 230]{PesinChicago}, i.e.  there is a continuous function $a$ on $\Lambda$ such that for any $x\in \Lambda$ the map $d_xf|_{T_xW^u(x)}$ is $a(x)$ times an isometry $I_x: T_xW^u(x)\to T_{f(x)}W^u(f(x))$,
\begin{equation}\label{E:au}
d_xf|_{T_xW^u(x)}=a(x)\:I_x,\quad x\in \Lambda.
\end{equation}
The last condition is true in particular for all diffeomorphisms $f$ that are conformal on $\Lambda$ in the sense of \cite[Definition 4.3.1]{Barreira}. Since $f$ is $C^{1+\alpha}$ and the spaces $E^u(x)$ vary H\"older continuously in $x\in \Lambda$, the function $a$ is H\"older continuous on $\Lambda$, cf. \cite{PesinChicago}.  The iterates of $f$ and $f^{-1}$ satisfy 
\begin{equation}\label{E:powers}
d_xf^n|_{T_x W^u(x)}=\left[\prod_{k=0}^{n-1} a(f^k(x))\right] I_{f^{n-1}(x)}\circ I_{f^{n-2}(x)}\circ \cdots \circ I_x
\end{equation}
and 
\begin{equation}\label{E:negativepowers}
d_xf^{-n}|_{T_x W^u(x)}=\left[\prod_{k=1}^n a(f^{-k}(x))\right]^{-1} I_{f^{-n}(x)}^{-1}\circ I_{f^{-n+1}(x)}^{-1}\circ \cdots \circ I_{f^{-1}(x)}^{-1}
\end{equation}
for any $n\geq 1$ and $x\in \Lambda$, as can be seen using the chain rule. Taking hyperbolicity into account, it follows that
\begin{equation}\label{E:naivebounds}
\underline{a}\leq |a|\leq \overline{a}
\end{equation}
on $\Lambda$ with constants $1<\underline{a}\leq \overline{a}<+\infty$. It seems convenient to express quantities in terms of the (H\"older continuous) potential 
\[\Phi(x):=-\log |a(x)|,\quad x\in \Lambda.\]

We observe the following consequences of the $f$-invariance of $\mu$ and the $u$-conformality of $f$.

\begin{theorem}\label{T:quasiinv}
Suppose that $f$ is $u$-conformal and $\mu$ is a Gibbs $u$-measure on $\Lambda$.
\begin{enumerate}
\item[(i)] For each $n\in \mathbb{Z}$ and $\varphi\in\overline{\mathcal{D}}(\mathcal{E})$ we have $\varphi\circ f^n\in \overline{\mathcal{D}}(\mathcal{E})$, and similarly with $\mathcal{D}(\mathcal{E})$ in place of $\overline{\mathcal{D}}(\mathcal{E})$. Setting 
\[\mathcal{E}^n(\varphi):=\mathcal{E}(\varphi\circ f^n)\] 
defines local Dirichlet forms $(\mathcal{E}^n, \overline{\mathcal{D}}(\mathcal{E}))$ and $(\mathcal{E}^n, \mathcal{D}(\mathcal{E}))$ on $L^2(\Lambda,\mu)$. If $\supp\mu=\Lambda$, then $(\mathcal{E}^n, \mathcal{D}(\mathcal{E}))$ is regular and strongly local.
\item[(ii)] For each $n\in \mathbb{N}$ and $\varphi\in \overline{\mathcal{D}}(\mathcal{E})$ we have 
\begin{align}\label{E:quasiinv}
\mathcal{E}^n(\varphi)=\int_\Lambda \left\|\nabla \varphi(x)\right\|_{T_{x} W^u(x)}^2 e^{-2\sum_{k=1}^{n}\Phi(f^{-k}(x))} \mu(dx)
\end{align}
and
\begin{align}\label{E:quasiinvneg}
\mathcal{E}^{-n}(\varphi)=\int_\Lambda \left\|\nabla \varphi(x)\right\|_{T_{x} W^u(x)}^2 e^{2\sum_{k=0}^{n-1}\Phi(f^k(x))} \mu(dx).
\end{align}
\item[(iii)] For each $n\in\mathbb{N}$ and $\varphi\in \overline{\mathcal{D}}(\mathcal{E})$ we have 
\begin{equation}\label{E:sandwichtenergy}
\underline{a}^{2n} \mathcal{E}(\varphi)\leq \mathcal{E}^n(\varphi)\leq \overline{a}^{2n}\:\mathcal{E}(\varphi),
\end{equation}
with $\underline{a}$ and $\overline{a}$ from (\ref{E:naivebounds}), and similarly for $\mathcal{E}^{-n}$ with $\overline{a}^{-2n}$ and $\underline{a}^{-2n}$.
\end{enumerate}
\end{theorem}

\begin{proof}
Suppose that $\varphi\in \overline{\mathcal{D}}_0(\mathcal{E})$. By the invariance of the global unstable manifolds and the chain rule we have
$d_x(\varphi\circ f^n)=d_{f^n(x)}\varphi\cdot d_xf^n$, and by (\ref{E:powers}), 
\begin{align}
\left\|d_x(\varphi\circ f^n)\right\|_{T_x^\ast W^u(x)}&=\left\|d_{f^n(x)}\varphi \right\|_{T_{f^n(x)}^\ast W^u(f^n(x))}  \prod_{k=0}^{n-1} |a(f^k(x))|\notag\\
&=\left\|d_{f^n(x)}\varphi \right\|_{T_{f^n(x)}^\ast W^u(f^n(x))} e^{-\sum_{k=0}^{n-1}\Phi(f^k(x))}, \quad x\in \Lambda.\notag
\end{align}
Clearly $\left\|\nabla \varphi(x)\right\|_{T_xW^u(x)}=\left\|d_x\varphi\right\|_{T_x^\ast W^u(x)}$ by duality. By (\ref{E:defquadform}) and the $f$-invariance of $\Lambda$, $\mu$ and the global unstable manifolds, the change of variable $y=f^n(x)$ yields
\begin{align}
\mathcal{E}(\varphi\circ f^n)&=\int_\Lambda \left\|d_{f^n(x)}\varphi\right\|_{T_{f^n(x)}^\ast W^u(f^n(x))}^2 e^{-2\sum_{k=0}^{n-1}\Phi(f^k(x))} \mu(dx)\notag\\
&=\int_\Lambda \left\|d_{y}\varphi\right\|_{T_{y}^\ast W^u(y)}^2 e^{-2\sum_{k=1}^{n}\Phi(f^{-k}(y))} \mu(dy).\notag
\end{align}
The bounds (\ref{E:naivebounds}) imply (\ref{E:sandwichtenergy}) for elements of $\overline{\mathcal{D}}_0(\mathcal{E})$ and that $\varphi\circ f^n\in \overline{\mathcal{D}}_0(\mathcal{E})$. As a consequence the closure of $(\mathcal{E}^{n}, \overline{\mathcal{D}}_0(\mathcal{E}))$ has domain $\overline{\mathcal{D}}(\mathcal{E})$. The proof for $(\mathcal{E}^n,\mathcal{D}(\mathcal{E}))$ is similar. The statements for $\mathcal{E}^{-n}$ follow using (\ref{E:negativepowers}). Standard estimates, \cite[Lemma 2.5, formula (2.3)]{Hino}, show that (\ref{E:sandwichtenergy}) holds for all elements of $\overline{\mathcal{D}}(\mathcal{E})$. 
\end{proof}

For any $n\in\mathbb{N}$ the carr\'e du champ operator $\Gamma^{n}$ of $\mathcal{E}^{n}$ satisfies
\begin{align}
\Gamma^{n}(\varphi)(x)&=\left\|\nabla \varphi(x)\right\|_{T_{x} W^u(x)}^2 e^{-2\sum_{k=1}^{n}\Phi(f^{-k}(x))}\notag\\
&=\Gamma(\varphi)(x)e^{-2\sum_{k=1}^{n}\Phi(f^{-k}(x))},\quad  \varphi \in \overline{\mathcal{D}}(\mathcal{E}),   \notag
\end{align}
and a similar identity holds for the carr\'e du champ operator $\Gamma^{-n}$ of $\mathcal{E}^{-n}$. Theorem \ref{T:quasiinv} may therefore be regarded as a \emph{quasi-invariance property} of the carr\'e du champ under $f$; the functions $ e^{-2\sum_{k=1}^{n}\Phi(f^{-k}(x))}$ and $e^{2\sum_{k=0}^{n-1}\Phi(f^k(x))}$ may be seen as \emph{conformal factors}. It follows directly from the proof of Theorem \ref{T:quasiinv} (or can be concluded from (\ref{E:sandwichtenergy})) that for any $\varphi \in \overline{\mathcal{D}}(\mathcal{E})$ we have
\begin{equation}\label{E:sandwichthecarre}
\underline{a}^{2n} \Gamma(\varphi)\leq \Gamma^{n}(\varphi)\leq \overline{a}^{2n}\:\Gamma(\varphi)\quad \text{$\mu$-a.e. on $\Lambda$,}
\end{equation}
and similarly for $\Gamma^{-n}$.

\begin{remark}\label{R:CM}\mbox{}
\begin{enumerate} 
\item[(i)] For the standard Dirichlet integral on Euclidean spaces both the Lebesgue measure and the carr\'e du champ $\Gamma(\varphi)=|\nabla\varphi|^2$ are invariant under translations. For the standard Dirichlet form on an abstract Wiener space the carr\'e du champ is 
invariant under translations in Cameron-Martin directions, and the Gaussian measure is quasi-invariant, see \cite[Chapter II, Section 2, in particular Theorem 2.4.4]{BH91}. In the situation of Theorem \ref{T:quasiinv} the measure $\mu$ is invariant under $f$, and the carr\'e du champ is quasi-invariant. 
\item[(ii)] By (\ref{E:sandwichthecarre}) the map $f$ may also be viewed as a quasi-isometry in a Dirichlet form sense (cf. \cite[Definition 2.17 and Lemma 2.18]{BBK06}, \cite[p.6]{BarlowMurugan}). 
\item[(iii)] Informally, the generator $\mathcal{L}^{n}$ of $\mathcal{E}^{n}$ is $\mathcal{L}^{n}\varphi=\diverg_\mu(e^{-2\sum_{k=1}^{n}\Phi\circ f^{-k}}\nabla \varphi)$, and 
similarly for the generator of $\mathcal{E}^{-n}$.
\end{enumerate} 
\end{remark}

\begin{examples}\mbox{}
\begin{enumerate}
\item[(i)] Recall the hyperbolic toral automorphisms $f_A$ defined in Examples \ref{Ex:classex} (i). For the special case that $n=2$ and 
\[A=\begin{pmatrix} 2 & 1\\ 1 & 1 \end{pmatrix}\] 
the map $f_A:\mathbb{T}^2\to \mathbb{T}^2$ is known as \textit{Arnold's cat map}, it is $u$-conformal with $a\equiv\lambda_u$ in (\ref{E:au}),
$\lambda_u=(3+\sqrt{5})/2>1$ being the larger eigenvalue of $A$.
\item[(ii)] The maps $f$ defining the Smale-Williams solenoids in Examples \ref{Ex:classex} (ii) are $u$-conformal with constant $a\equiv 2$ in (\ref{E:au}).
\end{enumerate}
\end{examples}

If, as in these examples, the function $a$ in (\ref{E:au}) is constant, then the preceding reduces to a simple rescaling. For simplicity we consider $(\mathcal{E},\mathcal{D}(\mathcal{E}))$, satements (i), (ii) and (iii) have analogous versions for $(\mathcal{E},\overline{\mathcal{D}}(\mathcal{E}))$.

\begin{corollary}\label{C:confirmation}
Suppose that $f$ is $u$-conformal and that $a$ is constant. Suppose further that $\mu$ is a Gibbs $u$-measure as in Theorem \ref{T:quasiinv} with $\supp\:\mu=\Lambda$. Let $n\in \mathbb{Z}$.
\begin{enumerate}
\item[(i)] For any $\varphi\in \mathcal{D}(\mathcal{E})$ we have 
$\mathcal{E}^{n}(\varphi)=a^{2n}\mathcal{E}(\varphi)$ and $\Gamma^{n}(\varphi)=a^{2n}\Gamma(\varphi)$.
\item[(ii)] The generator $\mathcal{L}^{n}$ of $\mathcal{E}^{n}$ has domain $\mathcal{D}(\mathcal{L})$. For any $\varphi \in \mathcal{D}(\mathcal{L})$ we have $\varphi\circ f^n \in \mathcal{D}(\mathcal{L})$ and 
\begin{equation}\label{E:multiple}
\mathcal{L}^{n}\varphi =a^{2n}\mathcal{L}\varphi.
\end{equation}
\item[(iii)] The symmetric Markov semigroup $(P^{n}_t)_{t>0}$ generated by $\mathcal{L}^{n}$ is given by $P_t^{n}=P_{a^{2n}t}$, $t>0$.
\item[(iv)] The symmetric Hunt diffusion process uniquely associated with $(\mathcal{E}^{n},\mathcal{D}(\mathcal{E}))$ in the sense of Theorem \ref{T:regular} is $((X_{a^{2n}t})_{t\geq 0},(\mathbb{P}^x)_{x\in \Lambda\setminus \mathcal{N}})$.
\end{enumerate}
\end{corollary}

Since $a>1$ it follows that, roughly speaking, for $n>0$ the expansion effect of $f^n$ makes the process $(X_{a^{2n}t})_{t\geq 0}$ run faster than $(X_t)_{t\geq 0}$, while for $n<0$ it slows the process down.

\begin{proof}
Statement (i) is clear from Theorem \ref{T:quasiinv} and implies (ii). Statement (iii) follows from (ii) and the spectral representation, (iv) from (iii) and (\ref{E:coincideae}).
\end{proof}

As a consequence of (\ref{E:multiple}) harmonicity properties are preserved under $f$, this is similar to the situation in the (complex) plane \cite[Section 4.1, Problem 4]{Ahlfors}. We recall the notion of superharmonicity in the Dirichlet form sense 
\cite{BBKT10, Chen}: Given a regular Dirichlet form $(\mathcal{E},\mathcal{D}(\mathcal{E}))$ on $L^2(\Lambda,\mu)$ and an open set $O\subset \Lambda$, a function $\varphi\in \mathcal{D}(\mathcal{E})$ is called \emph{superharmonic in $O$} if $\mathcal{E}(\varphi,\psi)\geq 0$ for all nonnegative $\psi\in \mathcal{D}(\mathcal{E})\cap C_c(O)$. 

\begin{corollary}
Let the hypotheses of Corollary \ref{C:confirmation} be in force. If $n\in \mathbb{Z}$ and $\varphi$ is superharmonic in an open set $O\subset \Lambda$ then $\varphi\circ f^n$ is superharmonic in $f^{-n}(O)$. 
\end{corollary}
\begin{proof}
For any nonnegative $\psi\in \mathcal{D}(\mathcal{E}) \cap C_c(f^{-n}(O))$ the function $\psi\circ f^{-n}$ is nonnegative and in $\mathcal{D}(\mathcal{E}) \cap C_c(O)$, so that
\begin{equation*}
\mathcal{E}(\varphi\circ f^{n},\psi)=\mathcal{E}^{n}(\varphi,\psi\circ f^{-n})=a^{2n}\mathcal{E}(\varphi,\psi\circ f^{-n})\geq 0.
\end{equation*}
\end{proof}

\section{Functions of zero energy}\label{S:zero}

We quote the following special case of a well-known fact about the existence of measurable partitions, see \cite[Theorem 2]{PesinSinai}, \cite[Lemma 3.1.1 and Section 6]{LedrappierYoung1} or \cite[Section 9.4]{BarreiraPesin2} for more general versions.

\begin{theorem}\label{T:PS82} Suppose that $\Lambda\subset M$ is a uniformly hyperbolic attractor for a $C^{1+\alpha}$ diffeomorphism $f:M\to M$,  $f|_\Lambda$ is topologically transitive and $\mu$ is the unique SRB measure on $\Lambda$.  

Then there are an $f$-invariant open subset $\Lambda'$ of $\Lambda$ with $\mu(\Lambda')=1$ and a measurable partition $\mathcal{P}_{\Lambda'}$ of $\Lambda'$ into connected open subsets $G$ of the global unstable manifolds. For $\hat{\mu}_{\mathcal{P}_{\Lambda'}}$-a.e. $G\in \mathcal{P}_{\Lambda'}$ the conditional measures $\mu_G$ have strictly positive densities; here $\hat{\mu}_{\mathcal{P}_{\Lambda'}}$ denotes the quotient probability measure of $\mathcal{P}_{\Lambda'}$.
\end{theorem}

Theorems \ref{T:PS82} and \ref{T:Rokhlin} yield an additional 'global' formula for $\mathcal{E}$. 

\begin{corollary}\label{C:addformula}
Let $f$, $\Lambda$, $\mu$ and $\mathcal{P}_{\Lambda'}$ be as in Theorem \ref{T:PS82}. Then we have 
\begin{equation}\label{E:globallyintegrated}
\mathcal{E}(\varphi)=\int_{\mathcal{P}_{\Lambda'}} D^{(\mu_G)}(\varphi) \:\hat{\mu}_{\mathcal{P}_{\Lambda'}}(dG),\quad \varphi\in \mathcal{D}_0(\mathcal{E}),
\end{equation}
where $D^{(\mu_G)}$ denotes the Dirichlet integral on $G\in \mathcal{P}_{\Lambda'}$ with respect to $\mu_G$ as defined in (\ref{E:Dirichletintegral}).
\end{corollary}
\begin{proof}
For any $\varphi\in \mathcal{D}_0(\mathcal{E})$ definitions (\ref{E:classgrad}) and (\ref{E:defquadform}), together with Theorem \ref{T:PS82} and (\ref{E:Dirichletintegral}), yield
\[
\mathcal{E}(\varphi)=\int_{\Lambda'}\left\|\nabla_{W^u(x)} \varphi(x)\right\|^2_{T_xW^u(x)}\mu(dx)=\int_{\mathcal{P}_{\Lambda'}}\int_{G} \left\|\nabla_{G} \varphi(x)\right\|^2_{T_xG}\mu_{G}(dx)\hat{\mu}_{\mathcal{P}_{\Lambda'}}(dG).
\]
\end{proof}

The right hand side of (\ref{E:globallyintegrated}) makes sense for  any bounded Borel function $\varphi$ on $\Lambda$ since for such $\varphi$ we have $\varphi|_{G}$ is in $L^2(G,\mu_{G})$. We therefore use (\ref{E:globallyintegrated}) to consistently define $\mathcal{E}$ to all bounded Borel $\varphi$.

\begin{corollary}\label{C:locallyconst}
Let $f$, $\Lambda$, $\mu$ and $\mathcal{P}_{\Lambda'}$ be as in Theorem \ref{T:PS82}.
If $\varphi:\Lambda\to \mathbb{R}$ is a bounded Borel function with $\mathcal{E}(\varphi)=0$, then $\varphi$ is constant on $\hat{\mu}_{\mathcal{P}_{\Lambda'}}$-a.e. $G\in \mathcal{P}_{\Lambda'}$.
\end{corollary}

\begin{proof} By (\ref{E:globallyintegrated}) we have $D^{(\mu_G)}(\varphi)=0$ for $\hat{\mu}_{\mathcal{P}_{\Lambda'}}$-a.e. $G\in \mathcal{P}_{\Lambda'}$, and by the positivity of the densities 
\[\int_{G} \left\|\nabla_{W^u(y)}\varphi(y)\right\|_{T_yW^u(y)}^2 m_{G}(dy)=0.\]
Since any such $G$ is a connected open subset of a Riemannian manifold of dimension $d_u$, a standard argument shows that $\varphi|_{G}$ must be constant: For any $y\in G$ we can find a local chart $(U,x)$ around $y$ such that
\[\int_{x(U)} |\nabla_{\mathbb{R}^{d_u}}(\varphi\circ x^{-1})(y)|^2\:dy\leq c\:\int_{U}\left\|\nabla_{W^u(y)}\varphi(y)\right\|^2_{T_yW^u(y)}m_{W^u(y)}(dy)=0,\]
hence $\varphi$ is constant on $U$.  
\end{proof}

We observe the following complement to Corollary \ref{C:locallyconst}.

\begin{theorem}\label{T:noLiouville} Suppose that $\Lambda\subset M$ is a uniformly hyperbolic attractor for a $C^{1+\alpha}$ diffeomorphism $f:M\to M$,  $f|_\Lambda$ is topologically transitive and $\mu$ is the unique SRB measure on $\Lambda$. Let $\Lambda'$ and $\mathcal{P}_{\Lambda'}$ be as in Theorem \ref{T:PS82}.
\begin{enumerate}
\item[(i)] There is a bounded Borel function $\varphi:\Lambda\to \mathbb{R}$ with $\mathcal{E}(\varphi)=0$ which is constant on each $G\in \mathcal{P}_{\Lambda'}$ but not $\mu$-a.e. constant on $\Lambda$.
\item[(ii)] Any Borel function which is constant on each individual global unstable leave is $\mu$-a.e. constant on $\Lambda$.
\item[(iii)] The function $\varphi$ as in (i) is not constant on each individual global unstable leave.
\end{enumerate}
\end{theorem}

Recall the notions of equivalence mod zero for partitions and $\sigma$-algebras, \cite[Definition 4.5]{Walters}. We write $\mathcal{B}^u$ and $\mathcal{B}(\mathcal{P}_{\Lambda'})$ for the $\sigma$-algebras consisting of all Borel sets that are unions of global unstable leaves respectively unions of elements of $\mathcal{P}_{\Lambda'}$. By $\pi(f)$ we denote the Pinsker $\sigma$-algebra of a given measure preserving map $f$, \cite[Section 4.10]{Walters}.

\begin{proof}
Statement (iii) is immediate from (i) and (ii). Since by construction the measurable partition $\mathcal{P}_{\Lambda'}$ is nontrivial, also $\mathcal{B}(\mathcal{P}_{\Lambda'})$ cannot be trivial mod zero. Consequently there is some $F\in \mathcal{B}(\mathcal{P}_{\Lambda'})$ with $0<\mu(F)<1$. The function $\varphi=\mathbf{1}_F$ is Borel measurable, and it is constant on each $G\in \mathcal{P}_{\Lambda'}$. Since each $G$ is an open subset of a global unstable leave, we have
$D^{\mu_G}(\varphi)=0$, and therefore $\mathcal{E}(\varphi)=0$ by (\ref{E:globallyintegrated}), what proves (i). To see (ii) note that under the stated hypotheses the measure preserving map $f$ on $(\Lambda,\mathcal{B}(\Lambda),\mu)$ is Bernoulli, \cite{Bowen75, Ledrappier}, and consequently the $\sigma$-algebra $\pi(f)$ is trivial mod zero, \cite[Theorems 4.30 and 4.34 and Section 4.10]{Walters}. Since $\pi(f)$ equals $\mathcal{B}^u$ mod zero, \cite[Theorem B and p. 536]{LedrappierYoung1}, also $\mathcal{B}^u$ is trivial mod zero. Suppose now that $\varphi$ is Borel, constant on each individual global unstable leave, but not $\mu$-a.e. constant on $\Lambda$. Then there are numbers $a<b$ such that both $\{\varphi\leq a\}$ and $\{\varphi\geq b\}$ are elements of $\mathcal{B}^u$ of positive measure, a contradiction.
\end{proof}

\begin{remark}\mbox{}
\begin{enumerate}
\item[(i)] We can fix a $\hat{\mu}_{\mathcal{P}_{\Lambda'}}$-null set $\mathcal{N}\subset \mathcal{P}_{\Lambda'}$ so that on all $G\in \mathcal{P}_\Lambda\setminus \mathcal{N}$ the measures $\mu_{G}$ are absolutely continuous and view the family $(L^2(G,\mu_{G}))_G$ with $G$ ranging over $\mathcal{P}_{\Lambda'}\setminus \mathcal{N}$ as a measurable field of Hilbert spaces. Then the space $L^2(\Lambda,\mu)$ is isometrically isomorphic to the direct integral  $L^2(\mathcal{P}_{\Lambda'},(L^2(G,\mu_{G}))_G, \hat{\mu}_{\mathcal{P}_{\Lambda'}})$, and its elements can be interpreted as a class of $\hat{\mu}_{\mathcal{P}_{\Lambda'}}$-square integrable sections $\varphi=(\varphi_{G})_G$ with elements $\varphi_{G}\in L^2(G,\mu_{G})$. Writing $\widetilde{\mathcal{D}}(\mathcal{E})$ for the space of all $\varphi\in L^2(\Lambda,\mu)$ such that $G\mapsto D^{\mu_G}(\varphi_G)\in L^1(\mathcal{P}_{\Lambda'},\hat{\mu}_{\mathcal{P}_{\Lambda'}})$, we can extend $\mathcal{E}$ further to a quadratic form $(\mathcal{E},\widetilde{\mathcal{D}}(\mathcal{E}))$ by setting $\mathcal{E}(\varphi):=\int_{\mathcal{P}_{\Lambda'}} D^{(\mu_G)}(\varphi_G) \:\hat{\mu}_{\mathcal{P}_{\Lambda'}}(dG)$, $\varphi \in \widetilde{\mathcal{D}}(\mathcal{E})$. By Proposition \ref{P:superpos} this is a Dirichlet form, clearly local, and $\overline{\mathcal{D}}(\mathcal{E})\subset \widetilde{\mathcal{D}}(\mathcal{E})$.
\item[(ii)] Since functions $\varphi$ as in Theorem \ref{T:noLiouville} (i) are in $\widetilde{\mathcal{D}}(\mathcal{E})$, 
Theorem \ref{T:noLiouville} (i) may be interpreted as an existence statement for nonconstant bounded harmonic functions in this larger domain. No $L^p$-Liouville theorem ($1\leq p\leq +\infty$) and no Liouville theorem for bounded energy finite functions holds, cf. \cite{Grigoryan87}, \cite[p. 319/320]{Grigoryanbook}. Theorem \ref{T:noLiouville} (iii) could be rephrased by saying that no foliated Liouville property (in a measurable sense) holds, see \cite{FZ03}.
\end{enumerate}
\end{remark}

\section{Hyperbolic attractors with singularities}\label{S:has}

We consider a more general class of hyperbolic attractors induced by maps with discontinuities, \cite{Pesin92}, \cite[Section 8]{CLP17}. The notation in this section follows \cite{Pesin92}, up to minor details.

Let $M$ be a smooth Riemannian manifold, $U\subset M$ a relatively compact open set and $N\subset U$ a closed subset. Let $f:U\setminus N\rightarrow U$ be a $C^2$ diffeomorphism onto its image. We define $N^+:=N\cup \partial U$
as the singularity set for $f$, and 
\[N^-:=\lbrace y\in U: \text{there are $z\in N^+$ and $z_n\in U\setminus N^+$ with $z_n\to z$ and $f(z_n)\to y$}\rbrace\]
as the singularity set for $f^{-1}$. Assume that $f$ is such that 
\begin{align}
\left\|d_x^2f \right\|&\leq c_1\,d(x, N^+)^{-\alpha_1} \qquad \text{for any $x\in U\setminus N$},\notag\\
\left\| d_x^2f^{-1}\right\|&\leq c_2\,d(x, N^-)^{-\alpha_2} \qquad \text{for any $x\in f(U\setminus N)$,}\notag
\end{align}
where $c_i>0$, $\alpha_i\geq0$, $i=1,2$, and $\left\|\cdot\right\|$ denotes the operator norm. A \emph{topological attractor with singularities} for $f$ is defined to be the compact set $\Lambda:=\overline{D}$ where
\begin{equation}\label{A:singattractor}
D:=\bigcap_{n\geq0} f^n(U^+)\quad \text{and}\quad U^+:=\lbrace x\in U: f^n(x)\notin N^+,  n=0, 1, 2,\dots\rbrace.
\end{equation}
The set $D$ is invariant under $f$ and $f^{-1}$, and hence so is $\Lambda$.

To define a hyperbolic structure for $f$ on the set $D$ it is covenient to apply cone techniques, see for instance \cite[Section 5.4]{BrinStuck} or \cite[Section 6.4]{KatokHasselblatt}. Given $x\in M$, $a>0$ and a subspace $P(x)\subset T_xM$, the \textit{cone} at $x$ around $P(x)$ with angle $\theta$ is the set \[C(x, P(x), \theta):=\lbrace v\in T_xM: \measuredangle(v, P(x))\leq \theta\rbrace.\] Here we write $\measuredangle(v, P):=\min_{w\in P} \measuredangle(v,w)$ for any $P\subset T_xM$, and we define $\measuredangle(P', P)$ for $P,P'\subset T_xM$ in a similar manner. A topological attractor with singularities $\Lambda$ is said to be a \textit{uniformly hyperbolic attractor with singularities} (or \emph{generalized hyperbolic attractor}) if there exist constants $c>0$, $\lambda \in (0,1)$, and $\theta(x)>0$ , $x\in U\setminus N^+$, together with subspaces $P^s(x),P^u(x)\subset T_xM$, $x\in U\setminus N^+$, of complementary dimension, such that the two families of \emph{stable} and \emph{unstable cones} 
\[C^s(x)=C^s(x, P^s(x), \theta(x))\quad \text{and}\quad C^u(x)=C^u(x, P^u(x), \theta(x))\] 
satisfy the following conditions:
\begin{enumerate}
\item[(i)] the angles $\measuredangle(C^s(x),C^u(x))$, $x\in U\setminus N^+$, are uniformly bounded away from zero, 
\item[(ii)] we have 
\begin{align}
df(C^u(x))&\subset C^u(f(x))\qquad \,\,\,\,\,\,\text{ for any $x\in U\setminus N^+$,}\notag\\
df^{-1}(C^s(x))&\subset C^s(f^{-1}(x))
\qquad \text{ for any $x\in f(U\setminus N^+)$},\notag 
\end{align}
\item[(iii)] for any $n\geq 1$ we have
\begin{align}
\left\|d_xf^n v\right\|_{T_{f^n(x)}M}&\geq  c\lambda^{-n} \left\| v \right\|_{T_xM} \qquad \text{for $v\in C^u(x)$ and $x\in U^+$},\notag\\
\left\| d_xf^{-n}v\right\|_{T_{f^{-n}(x)}M}&\geq c\lambda^{-n} \left\| v \right\|_{T_xM} \qquad \text{for $v\in C^s(x)$ and $x\in f^n(U^+)$}.\notag
\end{align}
\end{enumerate}
See \cite[Section 1.3]{Pesin92} or \cite[Section 8]{CLP17}. 

Now let $\Lambda$ be a generalized hyperbolic attractor; we continue to use  the above notation. By the standard arguments in \cite{Pesin77}, for any $x\in D$ the subspaces 
\[E^s(x)=\bigcap_{n\geq0} df^{-n} C^s(f^n(x)) \quad \text{and} \quad E^u(x)=\bigcap_{n\geq0} df^{n} C^u(f^{-n}(x))\]
form a splitting of the tangent space $T_xM=E^s(x)\oplus E^u(x)$ such that for any $n\geq0$
\begin{align}
\left\|d_xf^n v\right\|_{T_{f^n(x)}M}&\leq c\lambda^n \left\| v \right\|_{T_xM} \quad \quad\text{for $v\in E^s(x)$},\notag\\
\left\| d_xf^{-n}v\right\|_{T_{f^{-n}(x)}M}&\leq c\lambda^{n} \left\| v \right\|_{T_xM} \qquad \text{for $v\in E^u(x)$,}\notag
\end{align}
meaning that $D$ is a uniformly hyperbolic set contained in $\Lambda$, see \cite[p. 128]{Pesin92}. For any $\varepsilon>0$ and $\ell\geq 1$ we write
\begin{align*}
D^+_{\varepsilon, \ell}&:=\lbrace x\in \Lambda: d(f^n(x), N^+)\geq \ell^{-1} e^{-\varepsilon n},\, n\geq 0\rbrace,\\
D^-_{\varepsilon, \ell}&:=\lbrace x\in \Lambda: d(f^{-n}(x), N^-)\geq \ell^{-1} e^{-\varepsilon n},\, n\geq 0\rbrace,
\end{align*}
and 
\begin{equation}\label{E:Dminus}
D^+_{\varepsilon}:=\bigcup_{\ell\geq 1} D^+_{\varepsilon, \ell},\qquad D^-_{\varepsilon}:=\bigcup_{\ell\geq 1} D^-_{\varepsilon, \ell}.
\end{equation}
There is an adapted version of the Stable Manifold Theorem, \cite[Proposition 4]{Pesin92}:

\begin{theorem}\label{T:SMT3}
If $f$, $ D^+_\varepsilon$ and $ D^-_\varepsilon$ are as above, then for sufficiently small $\varepsilon>0$
local stable manifolds $V^s(x)$, $x\in D^+_\varepsilon$, and local unstable manifolds $V^u(x)$, $x\in D^-_\varepsilon$, exist and possess properties similar to those in Theorem \ref{T:SMT2} (i), (ii) and (iii). In particular, they are embedded submanifolds.
\end{theorem}

Henceforth we fix $\varepsilon>0$ at a sufficiently small value and write $D^-$ and $D_\ell^-$ instead of $D^-_\varepsilon$ and $D^-_{\varepsilon, \ell}$, respectively.
For any $\ell$ the manifolds $V^u(x)$ and the tangent spaces $T_xV^u(x)$ depend continuously on $x\in D_{\ell}^-$, \cite[Proposition 4]{Pesin92}, and we have $V^u(x)\subset D^-$ for any $x\in D^-$, see \cite[Proposition 5]{Pesin92}. Similarly as in (\ref{E:localtoglobal}), the global unstable manifolds are defined as
\begin{equation}\label{E:Wus}
W^u(x)=\bigcup_{n\geq0} {\hat{f}}^{n}(V^u(f^{-n}(x))),\quad x\in D^-,
\end{equation}
where we write ${\hat{f}}^n(E):=f^n(E\setminus N^+)$,  $E\subset \Lambda$.

Now let $\ell\geq 1$ be fixed. We write $B^u(z,\delta)$ for the open ball in $W^u(z)$ with center $z$ and radius $\delta$. By \cite[Proposition 7]{Pesin92}, there are $r^{(1)}_{\ell}>r^{(2)}_{\ell}>r^{(3)}_{\ell}>0$ such that for any $x\in D_{\ell}^-$ and any $z\in B(x,r^{(3)}_{\ell})\cap D^-_{\ell}$ the intersection $V^u(z)\cap W(x)$ of $V^u(z)$ and $W(x):=\exp_x\lbrace v\in E^s(x):\, \Vert v\Vert\leq r^{(1)}_{\ell}\rbrace$ is precisely a single point denoted by $[z,x]$ and, moreover, $B^u([z,x],r_\ell^{(2)})\subset V^u(z)$. Given $x\in D_{\ell}^-$ and $\delta\leq r^{(3)}_\ell$ we define the rectangle $\mathcal{R}_{\ell}(x,\delta)$ by 
\begin{equation}\label{E:rectangleGH}
\mathcal{R}=\mathcal{R}_{\ell}(x,\delta)=\bigcup_{z\in B(x,\delta)\cap D^-_{\ell}}B^u([z,x],r^{(2)}_{\ell}).
\end{equation}
For any fixed $\ell\geq 1$ the set $D_\ell^-$ is compact, and we can find points $x_1,...,x_{N_\ell}\in D_\ell^-$, neighborhoods $U_{x_i}$ of the $x_i$ in $U\setminus N$ and rectangles $\mathcal{R}_{x_i}$ such that $U_{x_i}\cap D_\ell^-\subset \mathcal{R}_{x_i}$ and we have 
\begin{equation}\label{E:coverDl}
D^-_{\ell}\subset \bigcup_{i=1}^{N_\ell} (U_{x_i}\cap D^-_{\ell})\subset\bigcup_{i=1}^{N_\ell} \mathcal{R}_{x_i}.
\end{equation}

\begin{remark}\label{R:GibbsGH}
Under suitable conditions one can ensure the existence of Gibbs $u$-measures, we quote \cite[Theorem 1]{Pesin92}: Suppose that $\Lambda$ is a uniformly hyperbolic attractor with singularities for the $C^2$ diffeomorphism $f$ and assume that there are a point $x\in D^-$ and constants $c>0$, $q>0$, $\varepsilon_0>0$ such that for any $0<\varepsilon\leq \varepsilon_0$ and $n\geq0$,
\begin{equation}\label{S:nullsing}
m_{V^u(x)}(V^u(x)\cap f^{-n}(U(\varepsilon, N^+)))\leq c\:\varepsilon^q,
\end{equation}
where $U(\varepsilon, N^+)$ is the $\varepsilon$-parallel neighborhood of $N^+$ in $M$. Then there is a Gibbs $u$-measure $\mu$ on $D^-\subset \Lambda$, it satisfies $\mu(D^-)=1$, and its conditional densities on the partition elements in (\ref{E:rectangleGH})  are H\"older continuous, uniformly bounded and uniformly bounded away from zero.
\end{remark}

Remark \ref{R:GibbsGH} allows to apply Theorem \ref{T:closable} and its consequences.

\begin{theorem}\label{T:has}
Let $f$ and $\Lambda$ be as in Remark \ref{R:GibbsGH} and let $\mu$ be a Gibbs $u$-measure on $D^-$. Then the statements in Theorem \ref{T:pha} (i), (ii) and (iii) hold with $A=D^-$ in place of $\Lambda$.
\end{theorem}

\begin{proof}
By (\ref{E:Dminus}), Theorem \ref{T:SMT3} and (\ref{E:Wus}) Assumption \ref{A:A1} is satisfied. The continuity of the dependence $x\mapsto T_xV^u(x)$ implies Assumption \ref{A:A2}. By (\ref{E:rectangleGH}) and (\ref{E:coverDl}) Assumption \ref{A:A3} holds with $A_\ell=D_\ell^-$. We may again apply Theorems \ref{T:closable} and \ref{T:regular} and Corollary \ref{C:semigroup} to obtain the stated results.
\end{proof}

\begin{remark}
In general $D^-$ is a proper subset of $\Lambda$, and tangent spaces in the unstable directions are defined only at points of $D^-$. But since $D^-$ has full measure, we have  $L^2(\Lambda,\mu)=L^2(D^-,\mu)$, so that the Dirichlet forms, Laplacians 
and semigroups in Theorem \ref{T:has} may be regarded as objects on $L^2(\Lambda,\mu)$, and in that sense 'on $\Lambda$'.
\end{remark}

We recall some examples for attractors with singularities.

\begin{examples}\mbox{}
\begin{enumerate}
\item[(i)] Let $I=(-1,1)$, $U=I\times I$ and $N=I\times\lbrace 0\rbrace$ and let $f:U\setminus N\rightarrow U$ be a map of the form $f(x,y):=(g(x,y),h(x,y))$, where $g,h$ are functions given by
\begin{align*}
g(x,y)&=(-a_2\vert y\vert^{\nu_0}+a_2x\,\text{sgn}\,y\vert y\vert^\nu+1)\,\text{sgn}\,y,\\
h(x,y)&=((1+a_1)\vert y\vert^{\nu_0}-a_1)\,\text{sgn}\,y,
\end{align*}
for constants $0<a_1<1$, $0<a_2<\frac{1}{2}$, $\nu>1$, $1/(1+a_1)<\nu_0<1$. The resulting attractor is the (\emph{geometric})\emph{ Lorenz attractor}; it is a generalized hyperbolic attractor. A more common definition of the Lorenz attractor is as the attractor (in ODE sense) for the nonlinear system
\begin{equation*}\label{L:lorenzeq}
\dot{x}=-\sigma x+\sigma y, \quad  \dot{y}=rx-y-xz \quad \text{and}\quad  \dot{z}=xy-bz 
\end{equation*}
for the particular parameters $\sigma=10$, $b=\frac{8}{3}$ and $r=28$, illustrated in Figure 2. Further details can be found in \cite[Section 5.2]{Pesin92}, \cite[Section 13.3]{HaPe03} and \cite[Section 2.2]{Kuznetsov}.

\begin{figure}
\centering
\includegraphics[height=5cm]{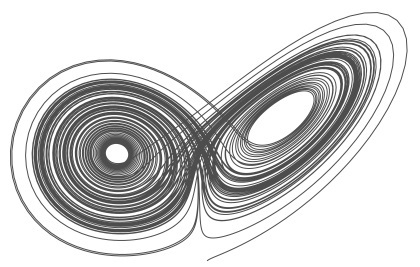}
\caption{The Lorenz attractor}
\end{figure}

\item[(ii)] The \emph{Lozi attractor},  \cite{Lozi,Pesin92,Young}, and the \emph{Belykh attractor}, \cite[p. 149]{Pesin92}, \cite{Sataev99} belong to the class of hyperbolic attractors with singularities.
\end{enumerate}
\end{examples}

\begin{remark}
Partially hyperbolic attractors with singularities $\Lambda=\overline{D}$ can similarly be constructed using (\ref{A:singattractor}), with the difference that $D$ must be a partially hyperbolic set. Stable and unstable manifolds can be constructed analogously, as well as Gibbs $u$-measures, \cite[Theorem 12]{Pesin92}. A geometric example of such attractors is provided in \cite[Section 5.2, Example 4]{Pesin92}. See also \cite{LS82}.
\end{remark}

\section{Attractors with nonuniformly hyperbolic structure}\label{S:nonuni}

Another generalization of uniformly hyperbolic dynamical systems is the theory of \emph{nonuniformly hyperbolic} dynamical systems (\emph{Pesin theory}); a detailed account can be found in \cite{BarreiraPesin, BarreiraPesin2}; see also \cite{Alves}, \cite{CLP17}, \cite{Pesin}.

Let $M$ and $f$ be as in Section \ref{S:hyperbolic}. Let $\Lambda\subset M$ be an $f$-invariant subset (not necessarily compact) and $\mu$ be an $f$-invariant Borel probability measure on $\Lambda$. The set $\Lambda$ is said to be \emph{nonuniformly hyperbolic} if for $\mu$-a.e. $x\in\Lambda$ there is a splitting of the tangent space \emph{stable} and \emph{unstable} subspaces 
\begin{equation}\label{N:nonusplitting}
T_xM=E^s(x)\oplus E^u(x)
\end{equation}
and there are measurable functions $c, \lambda,\varepsilon,k:\Lambda\rightarrow(0,\infty)$ with $\lambda(x)e^{\varepsilon(x)}<1$ such that for $\mu$-a.e. $x\in\Lambda$ we have $d_xf E^s(x)=E^s(f(x))$ and $d_xf E^u(x)=E^u(f(x))$,
\begin{align}
\left\|d_xf^n v\right\|_{T_{f^n(x)}M}&\leq c(x)\lambda(x)^n e^{\varepsilon(x)n} \left\| v \right\|_{T_xM} \qquad \text{for $v\in E^s(x)$ and $n\geq 0$},\notag\\
\left\| d_xf^{-n}v\right\|_{T_{f^{-n}(x)}M}&\leq c(x)\lambda(x)^n e^{\varepsilon(x)n} \left\| v \right\|_{T_xM} \qquad \text{for $v\in E^u(x)$ and $n\geq 0$,}\notag
\end{align}
$\measuredangle(E^s(x),E^u(x))\geq k(x)$, moreover, $\lambda(f(x))=\lambda(x)$, $\varepsilon(f(x))=\varepsilon(x)$ and 
\[c(f^n(x))\leq e^{\varepsilon(x)\vert n\vert}c(x)\quad \text{ and }\quad k(f^n(x))\geq e^{-\varepsilon(x)\vert n\vert}k(x),\quad n\in\mathbb{Z}.\]

There is a suitable version of the Stable Manifold Theorem, \cite[Chapter 7]{BarreiraPesin2} or \cite[Section 7.1]{BarreiraPesin}:
\begin{theorem}\label{T:SMTnonu} Let $M$ and $f$ be as in Section \ref{S:hyperbolic}, let $\mu$ be an $f$-invariant Borel probability measure on a nonuniformly hyperbolic set $\Lambda\subset M$ as above. Then, for $\mu$-a.e. $x\in \Lambda$ there are $C^r$ embedded submanifolds $V^s(x)$ and $V^u(x)$ that satisfy the properties (i) and (ii) in Theorem \ref{T:SMT2} and for any $y\in V^{u}(x)$ and $n\geq0$ we have 
\begin{equation}\label{E:expansionnonu}
d(f^{-n}(x),f^{-n}(y))\leq c(x)\lambda (x)^n e^{\varepsilon(x)n}d(x,y).
\end{equation} 
\end{theorem}

At $\mu$-a.e. $x\in \Lambda$ the global unstable manifold $W^u(x)$ containing $x$ can be defined as in (\ref{E:localtoglobal}), and we have $W^u(x)\subset \Lambda$.

Given $\ell\geq1$, a \emph{regular set} (or \emph{Pesin set}) of level $\ell$ is defined by
\[\Lambda_\ell:=\lbrace x\in\Lambda: c(x)\leq\ell\text{ and } k(x)\geq\ell^{-1}\rbrace.\]

Each regular set $\Lambda_\ell$ is compact. For any $\ell\geq1$ we have $\Lambda_\ell\subset\Lambda_{\ell+1}$, and
\begin{equation}\label{E:Lambdaell}
\Lambda=\bigcup_{\ell\geq1} \Lambda_\ell.
\end{equation}
It is not difficult to see that each $\Lambda_\ell$ is a uniformly hyperbolic set (but not necessarily $f$-invariant), see \cite[Section 4.3, Exercise 4.8]{BarreiraPesin}. For any fixed $\ell\geq1$ the stable and unstable subspaces $E^s(x)$ and $E^u(x)$ depend continuously on $x\in \Lambda_\ell$ and rectangles $\mathcal{R}=\mathcal{R}(x,\delta)$ can be be defined as in (\ref{E:rectanglePH}) with $\Lambda_\ell$ in place of $\Lambda$, finitely many rectangles will cover $\Lambda_\ell$.

\begin{remark}
SRB measures for topological attractors $\Lambda$ with nonuniformly hyperbolic structure can be constructed using \emph{effective hyperbolicity} \cite[Section 1.2]{CDP16}: If there are measurable invariant cone families and the set of effectively hyperbolic points has positive volume, then there is an SRB measure on $\Lambda$. See \cite[Section 1.2 and Theorem A]{CDP16} or \cite[Section 7, Theorem 7.2]{CLP17}.
\end{remark}

\begin{theorem}\label{T:nuha}
Let $M$ and $f$ be as in Section \ref{S:hyperbolic}, $\Lambda$ a topological attractor for $f$ with nonuniformly hyperbolic structure and endowed with an SRB measure $\mu$. Then the statements in Theorem \ref{T:pha} (i), (ii) and (iii) hold with $\Lambda$ as above.
\end{theorem}

\begin{proof}
By Theorem \ref{T:SMTnonu} a slightly weaker version of Assumption \ref{A:A1} is satisfied in the sense that we have a partition as stated there for a subset of $\Lambda$ of full $\mu$-measure. This still allows to define leafwise gradients as in (\ref{E:classgrad}) for $\mu$-a.e. $x\in \Lambda$, what is sufficient to define the quadratic forms (\ref{E:defquadform}).
Assumption \ref{A:A2} follows from (\ref{E:Lambdaell}) and the continuity of  $x\mapsto T_xV^u(x)$ on the sets $\Lambda_\ell$.
Assumption \ref{A:A3} holds with $A_\ell=\Lambda_\ell$. Lemma \ref{L:closable} and Theorem \ref{T:closable} remain applicable without changes.
\end{proof}

We recall a well-known nonuniformly hyperbolic example.

\begin{examples}
The classical H\'enon map $f_{a,b}:\mathbb{R}^2\rightarrow \mathbb{R}^2$ is defined by $f_{a,b}(x,y)=(1-ax^2+y,bx)$, where $a=1.4$ and $b=0.3$. The associated attractor is known as \emph{H\'enon attractor}, Figure 4. See \cite{Henon} and \cite[p. 187]{BarreiraPesin2}. The existence of SRB measure for certain H\'enon attractors was established in \cite{BeY93} based on the earlier work \cite{BeCa91}. See also \cite[Theorem 13.3.9]{BarreiraPesin2}.
\end{examples}

\begin{figure}
\centering
\includegraphics[height=5cm]{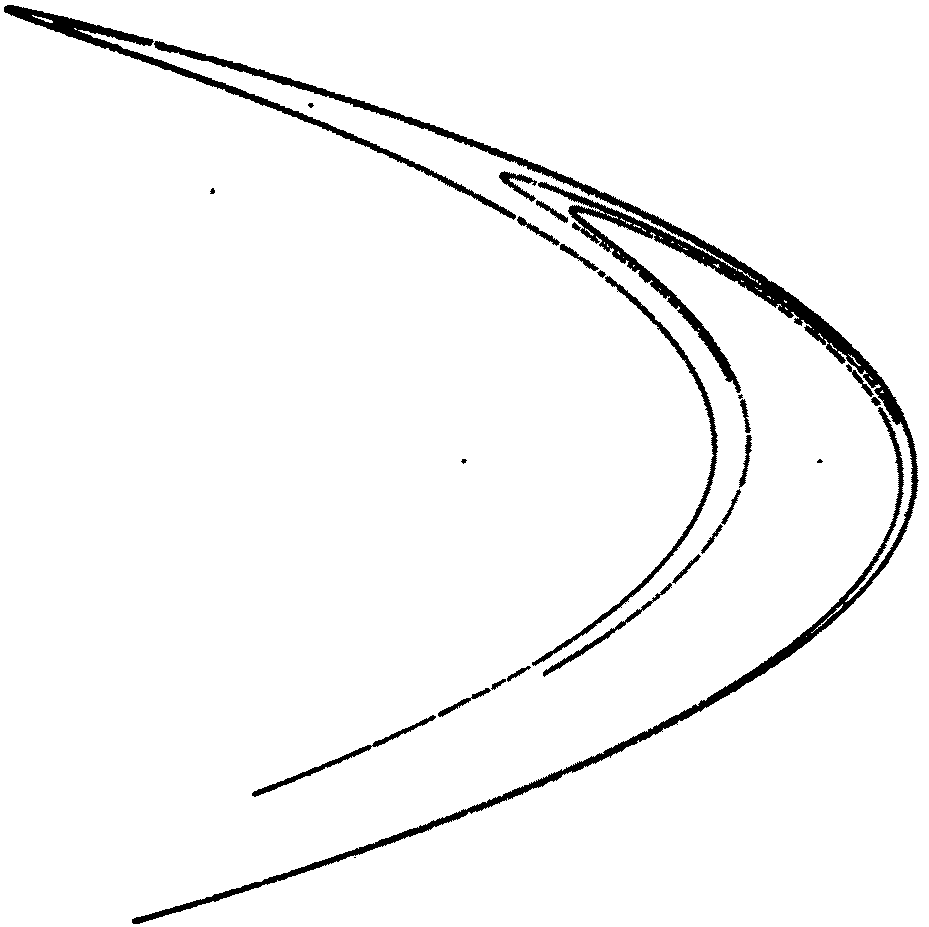}
\caption{The classical H\'enon attractor \cite{Henon}.}
\end{figure}

\appendix

\section{Disintegration and Rokhlin's theorem}\label{S:Rokhlin}

We recall a classical theorem due to Rokhlin \cite{Rokhlin}, further details can be found in \cite[Chapter 5]{Viana16}.

Let $(\mathcal{X}, \mathcal{A},\mu)$ be a probability space and $\mathcal{P}$ a partition of $\mathcal{X}$. A subset $\mathcal{Q}$ of $\mathcal{P}$ is called \emph{measurable} if the preimage of $\mathcal{Q}$ under the projection map $\pi: \mathcal{X}\rightarrow \mathcal{P}$ taking each point $x\in \mathcal{X}$ to the element $P(x)\in \mathcal{P}$ containing $x$ is measurable. The quotient probability measure $\hat{\mu}$ on $\mathcal{P}$ is defined as the pushforward of $\mu$ under $\pi$.  

A \textit{disintegration} of $\mu$ with respect to $\mathcal{P}$ is a family $\lbrace \mu_P\rbrace_{P\in \mathcal{P}}$ of probability measures on $\mathcal{X}$, called the \textit{conditional measures} (or \emph{conditional probabilities}), such that

\begin{enumerate}
\item[(i)] $\mu_P(P)=1$ for $\hat{\mu}$-a.e. $P\in \mathcal{P}$,
\item[(ii)] for any $A\in\mathcal{A}$ the function $P \mapsto \mu_P(A)$ is measurable,
\item[(iii)] we have 
\begin{equation}\label{E:Rokhlindisintegration}
\mu(A)=\int \mu_P(A)\, d\hat{\mu}(P) \quad \text{for any $A\in \mathcal{A}$.}
\end{equation}
\end{enumerate}
A partition $\mathcal{P}$ of $\mathcal{X}$ is called \textit{measurable} if there exists a sequence $(\mathcal{P}_n)_n$ of countable partitions $\mathcal{P}_n$ such that for $\mu$-a.e. $x\in \mathcal{X}$, $P_{n+1}(x)\subset P_{n}(x)$ for every $n\in \mathbb{N}$, and $P(x)=\bigcap_{n\in \mathbb{N}}P_{n}(x)$.

\begin{remark}
The partition into orbits of an ergodic measure theoretic dynamical system is not measurable, unless one orbit has full measure, 
\cite[Example 5.1.10]{Viana16}. The same happens for the partition of a uniformly hyperbolic attractor into global unstable manifolds, endowed with an ergodic probability measure, see for instance \cite[p. 513]{LedrappierYoung1}.
\end{remark}

We quote the following from \cite[Proposition 5.1.7 and Theorem 5.1.11]{Viana16}.

\begin{theorem}\label{T:Rokhlin} 
Let $\mu$ be a Borel probability measure on a compact metric space $\mathcal{X}$ and $\mathcal{P}$ a measurable partition of $\mathcal{X}$. There is a disintegration of $\mu$ with respect to $\mathcal{P}$. If $\lbrace \mu_P\rbrace_{P\in \mathcal{P}}$ and $\lbrace \mu_P'\rbrace_{P\in \mathcal{P}}$ are two disintegrations of $\mu$ with respect to $\mathcal{P}$ then 
$\mu_P=\mu_P'$ for $\hat{\mu}$-a.e. $P\in\mathcal{P}$.
\end{theorem}

\section{The geometric construction of Gibbs $u$-measures}\label{S:SRB}

For the convenience of the reader we recall the well-known geometric construction, \cite{CLP17}, of Gibbs $u$-measures on partially hyperbolic attractors, \cite[Propositions 2 and 3 and Theorem 4]{PesinSinai}. Suitable modifications of the arguments also cover the case of attractors with singularities, see \cite{Pesin92}.

We assume that $\Lambda$ is a hyperbolic attractor as in Section \ref{S:hyperbolic}; we use the same notation as there.  For any $z\in \Lambda$ we write $\mathcal{J}^u_z f$ to denote the Jacobian of $f$ along $W^u(z)$ (the modulus of the determinant of $d_zf\vert_{T_zW^u}:T_zW^u\to T_{f(z)}W^u$ with respect to the Riemannian metric). For any $n\geq 1$, any $x\in \Lambda$ and any $y\in W^u(x)$ we set $\varrho_n(x,y):=\prod_{j=1}^{n}\mathcal{J}^u_{f^{-j}(x)} f \big[\mathcal{J}^u_{f^{-j}(y)} f\big]^{-1}$.
There exists a constant $\kappa_0\geq 1$ such that for all $n\geq 1$ and for all $x\in \Lambda$ and $y\in V^u (x)$, we have 
\begin{equation}\label{E:boundeddistortion}
\kappa^{-1}_0\leq \varrho_n(x,y) \leq \kappa_0.
\end{equation}
This follows using a traditional argument, \cite[Lemma 3.3]{BonattiViana}: Since $f$ is $C^{1+\alpha}$, the function $z\mapsto \mathcal{J}^u_z f$ is $\alpha$-H\"older. Taking into account uniform hyperbolicity, 
\begin{equation}\label{E:key}
|\log \mathcal{J}^u_{f^{-j}(x)} f-\log \mathcal{J}^u_{f^{-j}(y)} f|\leq L\:d(f^{-j}(x),f^{-j}(y))^\alpha\leq LC^\alpha \lambda^{j \alpha} d(x,y)^\alpha
\end{equation}
with $L,C>0$ and $\lambda\in (0,1)$ as in Theorem \ref{T:SMT2}, and summing over $j\in \mathbb{N}$, we see that $\varrho_n(x,y)\leq \exp(LC^\alpha \lambda^\alpha d(x,y)^\alpha /(1-\lambda^\alpha))$, what implies (\ref{E:boundeddistortion}). In a similar way it follows that $|\varrho_n(x,y)-1|\leq C'\:d(x,y)^\alpha$ with $C'>0$ depending only on $\alpha$, $L$, $C$ and $\lambda$. Combined with (\ref{E:boundeddistortion}) this shows that 
\begin{equation}\label{E:Hoelderest}
|\varrho_n(x,y)-\varrho_n(x,z)|\leq C'K_0\:d(y,z)^\alpha,\quad \text{ for all $y,z\in V^u(x)$, $y \not= z$.}
\end{equation}
Applying Arzel{\`a}-Ascoli on a compact subset $K(x)$ of $V^u(x)$ containing $x$, we then see that
$\varrho(x,y):=\lim_n \varrho_n(x,y)$ with convergence uniform in $x\in \Lambda$ and $y\in K(x)$, \cite[Proposition 2]{PesinSinai}. Clearly $\varrho$ admits the same bounds as $\varrho_n$ in (\ref{E:boundeddistortion}), and using (\ref{E:Hoelderest}) it follows that for any $x$ the function $\varrho(x,\cdot) $ is $\alpha$-H\"older on $K(x)$. 

Now let $x\in \Lambda$ be fixed. Set $c_0:=1$ and $c_n:=\big[\prod_{k=0}^{n-1} \mathcal{J}^u_{f^k(x)}f\big]^{-1}$ and consider the measures on $f^n(V^u(x))$ defined by $\nu_n(dy):=c_n \varrho(f^n(x),y)\:m_{f^n(V^u(x))}(dy)$. Using the change of variable formula and balancing all cancellations, one can see that $\nu_n=f_\ast^n\nu_0$ for any $n$, \cite[Proposition 3 and Section 2.6]{PesinSinai}. Similarly as in the Krylov-Bogoliubov theorem one can now define averages 
\[\mu_n:=\frac{1}{n}\sum_{k=0}^{n-1} \nu_k,\]
and since $\nu_0$ is a finite measure, Prohorov's theorem implies that $(\mu_n)_n$ has a weak subsequential limit $\mu$, which obviously is $f$-invariant. If now $\mathcal{R}\subset \Lambda$ is a rectangle as in (\ref{E:rectanglePH}), then the weak convergence, together with disintegration on $\mathcal{R}$ and a uniform convergence argument for the densities of the conditional measures show that for $\hat{\mu}_{\mathcal{P}_\mathcal{R}}$-a.e. $V^u=V^u(x)$ in $\mathcal{P}_\mathcal{R}$ we have 
\[\mu_{V^u(x)}(dy)=\varrho(x)^{-1} \varrho(x,y) \:m_{V^u(x)}(dy)\]
with a renormalization constant $\varrho(x):=\int_{V^u(x)}\varrho(x,y)m_{V^u(x)}(dy)$,
\cite[Theorem 4, Lemma 13 and Section 2.7]{PesinSinai}. In particular, $\mu$ is a Gibbs $u$-measure with uniformly bounded H\"older continuous densities $\varrho_{V^u(x)}=\varrho(x)^{-1} \varrho(x,\cdot)$.

\section{Dirichlet integrals on weighted manifolds}\label{S:manifolds}

We collect some facts on Dirichlet integrals on weighted manifolds of low regularity and with uniformly bounded densities.

Let $M$ be a $C^1$-Riemannian manifold, let $m_M$ denote the Riemannian volume and $\nabla_M$ the gradient operator on $M$. 
Suppose further that $\varrho_M$ is a bounded and strictly positive continuous function on $M$ and consider the measure $\mu_M=\varrho_M\cdot m_M$
with density $\varrho_M$. The Riemannian volume itself corresponds to the choice $\varrho_M\equiv\mathbf{1}$.

Under these assumptions one can still introduce a kind of classical divergence operator. Consider the space 
\[\varrho_M^{-1} C_c^1(M, TM):=\left\lbrace \varrho_M^{-1} w: w\in C_c^1(M, TM)\right\rbrace,\]
where $\varrho_M^{-1}$ denotes the reciprocal of $\varrho_M$, and set 
\[\diverg_{\mu_M} v:=\varrho_M^{-1} \diverg_M (\varrho_M v),\quad v \in \varrho_M^{-1} C_c^1(M, TM),\]
where $\diverg_M$ is the usual divergence on $C^1$-vector fields. The following version of the divergence theorem is immediate from a $C^1$-version of the classical one, \cite[Theorem 3.5]{Grigoryanbook}.

\begin{corollary}\label{C:divergencetheorem}
Let $M$ and $\mu_M$ be as before. Then for any $\varphi\in C^1(M)$ and $v\in \varrho_M^{-1} C_c^1(M, TM)$ we have
\begin{equation}\label{E:divergencetheorem}
\int_M (\diverg_{\mu_M} v)\varphi\:d\mu_M=-\int_M\left\langle v,\nabla_M \varphi\right\rangle_{TM}d\mu_M.
\end{equation}
\end{corollary}

We make $\varrho_M^{-1} C_c^1(M, TM)$ a locally convex space by saying that a sequence $(v_j)_j \subset \varrho_M^{-1} C_c^1(M, TM)$ converges to some $v\in \varrho_M^{-1} C_c^1(M, TM)$ if all $v_j$ have support in a compact set $K\subset M$ and 
\[\lim_{j\to \infty}\left\|\varrho_M(v_j- v)\right\|_{C^1(M,TM)}=0,\]
where $\left\|w\right\|_{C^1(M,TM)}:=\sup_{x\in M}\left\|w(x)\right\|_{T_xM}$, $w\in C^1(M,TM)$. 

We write $(\varrho_M^{-1} C_c^1(M,TM))^\ast$ for the topological dual of the space $\varrho_M^{-1} C_c^1(M,TM)$. A sequence $(g_j)_j$ in this dual converges to an element $g$ if $\lim_{j\to \infty} g_j(v)=g(v)$ for all $v\in \varrho_M^{-1} C_c^1(M,TM)$.

Since for any $g\in L^2(M,TM,\mu_M)$ and $v\in \varrho_M^{-1} C_c^1(M,TM)$ supported in a compact set $K$ we have 
\begin{equation}\label{E:L2aredist}
|\int_M \left\langle v,g\right\rangle_{TM}\:d\mu_M|\leq \left\|\varrho_M v\right\|_{C^1(M,TM)}\:\sup_{x\in K}|\varrho_M(x)|^{-1}\:(\mu_M(K))^{1/2}\left\|g\right\|_{L^2(M,TM,\mu_M)},
\end{equation}
it follows that by
\[g(v):=\int_M \left\langle v,g\right\rangle_{TM} d\mu_M\]
we can assign an element of $(\varrho_M^{-1} C_c^1(M,TM))^\ast$ to $g$. By density this
assignment is injective and consequently it is justified to write again $g$ for this element of the dual space.
In this sense we have  
\[L^2(M,TM,\mu_M)\subset (\varrho_M^{-1} C_c^1(M,TM))^\ast.\] 
By (\ref{E:L2aredist}) convergence in $L^2(M,TM,\mu_M)$ implies convergence in $(\varrho_M^{-1} C_c^1(M,TM))^\ast$.

Given $\varphi\in L^1_{\loc}(M,\mu_M)$ we define its 'distributional' gradient $\nabla_M\varphi$ as an element of $(\varrho_M^{-1} C_c^1(M,TM))^\ast$ by 
\begin{equation}\label{E:distgrad}
\nabla_M\varphi(v):=-\int_M \varphi \diverg_{\mu_M}v\:d\mu_M,\quad v\in \varrho_M^{-1} C_c^1(M,TM).
\end{equation}
By Corollary \ref{C:divergencetheorem} this distributional gradient becomes the usual gradient if $\varphi\in C^1(M)$.

We consider the Sobolev space of (classes of) square integrable functions with square integrable gradients,
\begin{equation}\label{E:Sobospace}
W^{1,2}(M,\mu_M)=\left\lbrace \varphi\in L^2(M,\mu_M):\ \int_M \left\|\nabla_M \varphi(x)\right\|_{T_x M}^2\mu_M(dx)<+\infty\right\rbrace
\end{equation}
and define the Dirichlet integral for (classes of) square integrable functions by
\begin{equation}\label{E:Dirichletintegral}
D^{(\mu_M)}(\varphi):=\begin{cases} \int_M \left\|\nabla_M \varphi(x)\right\|_{T_xM}^2 \mu_M(dx) &\ \text{if $\varphi \in  W^{1,2}(M,\mu_M)$},\\ +\infty &\ \text{if $\varphi\in L^2(M,\mu_M)\setminus W^{1,2}(M,\mu_M)$}.\end{cases}
\end{equation}
By polarization $D^{(\mu_M)}$ defines a nonnegative definite symmetric bilinear form on $W^{1,2}(M,\mu_M)$, and we endow 
$W^{1,2}(M,\mu_M)$ with the scalar product 
\[(\varphi,\psi)\mapsto \left\langle \varphi,\psi\right\rangle_{L^2(M,\mu_M)}+D^{(\mu_M)}(\varphi,\psi).\]
The following fact can now be seen as in the classical case, \cite[Lemma 4.3]{Grigoryanbook}.
\begin{proposition}\label{P:closed}
The space $W^{1,2}(M,\mu_M)$ is Hilbert, and $(D^{(\mu_M)},W^{1,2}(M,\mu_M))$ is a Dirichlet form on $L^2(M,\mu_M)$.
\end{proposition}
\begin{proof} Suppose that $(\varphi_j)_j$ is Cauchy in $W^{1,2}(M,\mu_M)$.
We have $\lim_{j\to \infty} \varphi_j=\varphi$ in $L^2(M,\mu_M)$; we also have $\lim_{j\to \infty} \nabla_M \varphi_j=g$ in $L^2(M,TM,\mu_M)$ and therefore in $(\varrho_M^{-1} C_c^1(M,TM))^\ast$. But (\ref{E:distgrad}) implies that $\lim_{j\to \infty} \nabla_M \varphi_j=\nabla_M \varphi$ in $(\varrho_M^{-1} C_c^1(M,TM))^\ast$, what implies that $g=\nabla_M \varphi$.
\end{proof}

\begin{remark}\label{R:saLaplacianM}\mbox{}
\begin{enumerate}
\item[(i)] The generator $(\mathcal{L}^{(\mu_M)},\mathcal{D}(\mathcal{L}^{(\mu_M)}))$ of $(D^{(\mu_M)},W^{1,2}(M,\mu_M))$ 
is a non-positive definite self-adjoint operator on $L^2(M,\mu_M)$ and may be viewed as a Laplacian on $M$. Informally, 
$\mathcal{L}^{(\mu_M)}\varphi=\mathcal{L}^{(m_M)}\varphi+\left\langle \varrho_M^{-1}\nabla_M \varrho_M,\nabla_M\varphi\right\rangle_{TM}$,
where $\mathcal{L}^{(m_M)}$ denotes the usual Laplacian. If $M$ is of class $C^2$ and $\varrho_M\equiv \mathbf{1}$ then $(\mathcal{L}^{(m_M)},\mathcal{D}(\mathcal{L}^{(m_M)}))$ is a self-adjoint extension of the classical Laplace-Beltrami operator on $C^2_c(M)$.
\item[(ii)] If $\mathring{W}^{1,2}(M,\mu_M)$ denoted the closure of $C_c^1(M)$ in $W^{1,2}(M,\mu_M)$, then the quadratic form $(D^{(\mu_M)},\mathring{W}^{1,2}(M,\mu_M))$ is a strongly local regular Dirichlet form on $L^2(M,\mu_M)$. In general $\mathring{W}^{1,2}(M,\mu_M)$ may be smaller than $W^{1,2}(M,\mu_M)$. The generator of this form is referred to as the \emph{Dirichlet Laplacian} on $M$. The symmetric Hunt diffusion process on $M$ it generates is a \emph{distorted Brownian motion}, see for instance \cite{AH-KS77, Fukushima81, Kaimanovich89, TrutnauShin}.
\end{enumerate}
\end{remark}

\section{Superposition of closable forms}\label{S:superpos}

The following is a special case of \cite[Theorem 1.2]{AR90}, see also \cite[Chapter V, Proposition 3.1.1]{BH91} or \cite[Section 3.1. (2$^\circ$)]{FOT94}. For details on measurable fields of Hilbert spaces and their direct integrals see  \cite[Part II, Chapter 1]{Dix} or \cite[Chapter IV, Section 8]{Tak}.
\begin{proposition}\label{P:superpos}
Let $(Y,\mathcal{Y},\hat{\mu})$ be a $\sigma$-finite measure space and let $(H_y,\left\langle\cdot,\cdot\right\rangle_{H_y})_{y\in Y}$ be a measurable field of Hilbert spaces $H_y$ on $Y$. Suppose that there are quadratic forms $(\mathcal{Q}^{(y)},C_y)$ on $H_y$, $y\in Y$, respectively, and 
that for $\hat{\mu}$-a.e. $y\in Y$ the space $C_y$ is dense in $H_y$. For any $y\in Y$ and $w\in H_y$ set 
\[Q^{(y)}(w):=\begin{cases} Q^{(y)}(w) & \ \text{if $w\in C_y$}\\ +\infty & \ \text{otherwise}\end{cases}\]
and consider the space
\[\mathcal{C}:=\big\lbrace v=(v_y)_{y\in Y}\in L^2(Y,(H_y)_{y\in Y},\hat{\mu}): \text{the map $y\mapsto \mathcal{Q}^{(y)}(v_y)$ is $\hat{\mu}$-integrable} \big\rbrace.\]
If for $\hat{\mu}$-a.e. $y\in Y$ the form 
$(\mathcal{Q}^{(y)},C_y)$ is closable (closed) on $H_y$, then the quadratic form $(\mathcal{Q},\mathcal{C})$, defined by 
$\mathcal{Q}(v)=\int_Y \mathcal{Q}^{(y)}(v_y)\hat{\mu}(dy)$, $v\in\mathcal{C}$,
is closable (closed) on $L^2(Y,(H_y)_{y\in Y},\hat{\mu})$.
\end{proposition}


\begin{thebibliography}
\normalsize
\bibitem{Ahlfors}
L.V. Ahlfors, \emph{Complex Analysis. An Introduction to the Theory of Analytic Functions of One Complex Variable}, 3rd edition, McGraw-Hill, New York, 1979.

\bibitem{AH-KS77}
S. Albeverio, R. Høegh-Krohn, L. Streit, \emph{Energy forms, Hamiltonians, and distorted Brownian paths}, J. Math. Phys. {\bf 18}(5) (1977), 907--917.

\bibitem{AR90}
S. Albeverio, M. R\"ockner, \emph{Classical Dirichlet forms on topological vector spaces - closability and a Cameron-Martin formula}, J. Funct. Anal. {\bf 88} (1990), 395--436.

\bibitem{AlonsoRuiz18}
P. Alonso-Ruiz, \emph{Explicit formulas for heat kernels on diamond fractals}, Comm. Math. Phys. {\bf 364}(3) (2018),
1305--1326.

\bibitem{AlonsoRuiz21}
P. Alonso-Ruiz, \emph{Heat kernel analysis on diamond fractals}, 
Stoch. Proc. Appl. {\bf 131} (2021), 51--72.

\bibitem{ARHTT}
P. Alonso-Ruiz, M. Hinz, A. Teplyaev, R. Trevino, \emph{Canonical diffusions on the pattern spaces of aperiodic Delone sets}, preprint (2018), arXiv:1801.08956.

\bibitem{Alves}
J. F. Alves, \emph{Nonuniformly hyperbolic attractors: Geometric and probabilistic aspects}, Springer Monographs in Mathematics, Springer International Publishing, (2020).

\bibitem{Anosov67}
D. V. Anosov, \emph{Geodesic flows on closed Riemannian manifolds of negative curvature}, Trudy Mat. Inst.
Steklov. {\bf 90} (1967), 3--210.

\bibitem{A67}
D. V. Anosov, \emph{Tangential fields of transversal foliations in y-systems}, Math.
Notes, {\bf 2} (1967), no. 5, 818--823.


\bibitem{BBK06}
M. Barlow, R. F. Bass, T. Kumagai, \emph{Stability of parabolic Harnack inequalities on metric measure spaces},
J. Math. Soc. Japan {\bf 58}(2), 485--519.

\bibitem{BBKT10}
M. Barlow, R. F. Bass, T. Kumagai, A. Teplyaev, \emph{Uniqueness of Brownian motion on Sierpi\`nski carpets},
J. Eur. Math. Soc. (JEMS) 12 (2010), no. 3, 655--701.

\bibitem{BarlowMurugan}
M. Barlow, M. Murugan, \emph{Stability of the elliptic Harnack inequality}, Ann. Math. {\bf 187} (2018), 1--47.

\bibitem{Barreira}
L. Barreira, \emph{Dimension and Recurrence in Hyperbolic Dynamics}, Progress in Math. Vol. 272, Birkh\"auser, Basel, 2008.

\bibitem{BarreiraPesin}
L. Barreira, Ya. Pesin, \emph{Introduction to Smooth Ergodic Theory}, Grad. Studies in Math. Vol. 148, Amer. Math. Soc., Providence, 2013.

\bibitem{BarreiraPesin2}
L. Barreira, Ya. Pesin, \emph{Nonuniform hyperbolicity: Dynamics of systems with
nonzero Lyapunov exponents}, Cambridge University Press, (2007).

\bibitem{Belykh}
V. M. Belykh, \emph{Qualitative methods of the theory of nonlinear oscillations in point systems}, Gorki University Press (1980).

\bibitem{BeCa91}
M. Benedicks, L. Carleson, \emph{The dynamics of the H\'enon map},
Ann. Math. {\bf 133} (1991), 73--169.

\bibitem{BeY93}
M. Benedicks, L.-S. Young, \emph{Sinai–Bowen–Ruelle measure for certain H\'enon maps},
Invent. Math. 112 (1993), 541--576.

\bibitem{BG68}
R.M. Blumenthal, R.K. Getoor, \emph{Markov Processes and Potential Theory},
Acad. Press, New York, 1968.

\bibitem{BonattiViana}
C. Bonatti, M. Viana, \emph{SRB-measures for partially hyperbolic attractors whose central direction 
is mostly contratcting}, Israel J. Math. {\bf 115} (2000), 157--193. 

\bibitem{BH91}
N. Bouleau, F. Hirsch, \emph{Dirichlet Forms and Analysis on Wiener Space},
deGruyter Studies in Math. 14, deGruyter, Berlin, 1991.


\bibitem{Bowen74}
R. Bowen, \emph{Some systems with unique equilibrium states}, Math. Systems Theory {\bf 8} (1974), 193--202.

\bibitem{Bowen75} 
R. Bowen, \textit{Equilibrium States and the Ergodic Theory of Anosov Diffeomorphisms}, Lect. Notes in Math. {\bf 470}, Springer-Verlag, Berlin-New York, 1975.

\bibitem{BrinStuck}
M. Brin, G. Stuck, \emph{Introduction to Dynamical Systems}, 
Cambridge Univ. Press, Cambridge, 2002.

\bibitem{BDPP08}
K. Burns, D. Dolgopyat, Y. Pesin, M. Pollicott, \emph{Stable ergodicity for partially hyperbolic attractors
with negative central exponents}, J. Mod. Dyn. {\bf 2}(1) (2008), 1--19.

\bibitem{BG05}
K. Burns, M. Gidea, \emph{Differential Geometry and Topology: With a View to Dynamical Systems}, Chapman \& Hall/CRC Press, Boca Raton, FL, 2005.

\bibitem{Candel03}
A. Candel, \emph{The harmonic measures of Lucy Garnett}, Adv. Math. {\bf 176} (2003) 187--247.

\bibitem{CandelConlonI}
A. Candel, L. Conlon, \emph{Foliations I}, Grad. Studies in Math. Vol. 23, Amer. Math. Soc., Providence, 2000

\bibitem{CandelConlonII}
A. Candel, L. Conlon, \emph{Foliations II}, Grad. Studies in Math. Vol. 60, Amer. Math. Soc., Providence, 2003

\bibitem{Chen}
Z.-Q. Chen, \emph{On notions of harmonicity}, Proc. Amer. Math. Soc. {\bf 137} (2009), 3497--3510.

\bibitem{CDP16}
V. Climenhaga, D. Dolgopyat Y. Pesin, \emph{Non-stationary non-uniform hyperbolicity: SRB measures for dissipative maps}, Comm. Math. Phys. 346 (2016), no. 2, 553–602. 

\bibitem{CLP17} 
V. Climenhaga, S. Luzzatto, Y. Pesin, \textit{The geometric approach for constructing Sinai-Ruelle-Bowen measures}, J. Stat. Phys. {\bf 166} (2017), no. 3-4, 467--493. 

\bibitem{Coudene} 
Y. Coudene, \emph{Pictures of hyperbolic dynamical systems}, Notices Amer. Math. Soc. {\bf 53} (2006), 223--253.

\bibitem{Da89}
E.B. Davies, \emph{Heat Kernels and Spectral Theory}, Cambridge Tracts in Math. 92, Cambridge Univ. Press, Cambridge, 1989.


\bibitem{Dix}
J. Dixmier, \emph{Von Neumann Algebras},
North-Holland Math. Lib. 27, North-Holland, Amsterdam, 1981.


\bibitem{FZ03}
R. Feres, A Zeghib, \emph{Dynamics on the space of harmonic functions and the foliated Liouville problem}, Ergod. Th. Dynam. Sys. {\bf 25} (2005), 503--516.

\bibitem{FH19}
T. Fisher, B. Hasselblatt, \emph{Hyperbolic flows}, Z\"{u}rich Lectures in Advanced Mathematics, European Mathematical Society (EMS), Z\"{u}rich, 2019.

\bibitem{Fukushima81}
M. Fukushima, \emph{On a stochastic calculus related to Dirichlet forms and distorted
Brownian motions. New stochastic methods in physics.} Phys. Rep. {\bf 77} (3) (1981),
255--262.

\bibitem{FOT94}
M. Fukushima, Y. Oshima and M. Takeda, \emph{Dirichlet forms and symmetric Markov processes},
deGruyter, Berlin, New York, 1994.

\bibitem{Garnett83}
L. Garnett, \emph{Foliations, the ergodic theorem and Brownian motion}, J. Funct. Anal. {\bf 51} (1983), 285--311.

\bibitem{Grigoryan87}
A. Grigor’yan, \emph{On Liouville theorems for harmonic functions with finite Dirichlet
integral}, Matem. Sbornik {\bf 132}(4) (1987), 496--516. Engl. transl.: Math. USSR Sbornik {\bf 60}(2) (1988), 485--504.

\bibitem{Grigoryanbook}
A. Grigor'yan, \emph{Heat kernel and Analysis on Manifolds},
AMS/IP Stud. in Adv. Math. 47, Amer. Math. Soc., 2009.

\bibitem{Hamenstaedt97}
U. Hamenst\"adt, \emph{Harmonic measures for compact negatively curved manifolds},
Acta Math. {\bf 178} (1997), 39--107.

\bibitem{H17}
B. Hasselblatt, \emph{Ergodic Theory and Negative Curvature}, Springer Lecture Notes series. 2017.

\bibitem{HaPe03}
B. Hasselblatt and A. Katok, \emph{A First Course in Dynamics. With a Panorama of Recent
Developments}, Cambridge University Press, New York, 2003.

\bibitem {HaPe} 
B. Hasselblatt and Y. Pesin, \emph{Partially hyperbolic dynamical systems}, Handbook of dynamical systems. Vol. 1B, Elsevier B. V., 2006, pp. 1--55.


\bibitem{HaPeSc}
B. Hasselblatt, Ya. Pesin, J. Schmeling, \emph{Pointwise hyperbolicity implies uniform
hyperbolicity}, Disc. Cont. Dyn. Sys. {\bf 34} (7) (2014), 2819--2827.

\bibitem{Henon}
M. H\'enon, \emph{A two dimensional mapping with a strange attractor}, Comm. Math. Phys. {\bf 50}
 (1976), 69--77.

\bibitem{HewittRoss}
E. Hewitt, K.A. Ross, \emph{Abstract Harmonic Analysis. Volume I: Structure of Topological Groups
Integration Theory Group Representations}, 2nd edition, Springer, New York, 1979.

\bibitem{HewittRoss2}
E. Hewitt, K.A. Ross, \emph{Abstract Harmonic Analysis. Volume II: Structure and Analysis for Compact Groups. Analysis on Locally Compact Abelian Groups}, Springer, Berlin, 1970.



\bibitem{Hino}
M. Hino, \emph{Energy measures and indices of Dirichlet forms, with applications to derivatives on
soms fractals}, Proc. London Math. Soc. {\bf 100}(3) (2010), 269--302.


\bibitem{HRT13}
M. Hinz, M. R\"ockner, A. Teplyaev, \emph{Vector analysis for  Dirichlet forms and quasilinear PDE and SPDE on metric measure spaces}, Stoch. Proc. Appl. {\bf 123}(12) (2013), 4373--4406.

\bibitem{Kajino20a}
N. Kajino, \emph{The Laplacian on some self-conformal fractals and Weyl’s asymptotics for its eigenvalues:
A survey of the analytic aspects}, preprint (2020), arXiv:2001.07010.

\bibitem{Kajino20b}
N. Kajino, \emph{The Laplacian on some self-conformal fractals and Weyl’s asymptotics for its eigenvalues:
A survey of the ergodic-theoretic aspects}, preprint (2020), arXiv:2001.11354.

\bibitem{KatokHasselblatt}
A. Katok, B. Hasselblatt, \emph{Introduction to the Modern Theory of Dynamical Systems},
Cambridge Univ. Press, Cambridge, 1995.


\bibitem{Kigami01}
J. Kigami, \emph{Analysis on Fractals}, Cambridge Univ. Press, Cambridge, 2001.


\bibitem{Kuznetsov}
S.P Kuznetsov, \emph{Hyperbolic Chaos - A Physicist's View}, Springer, 2012.


\bibitem{L90}
F. Ledrappier, \emph{Harmonic measures and Bowen-Margulis measures},
Israel J. Math. {\bf 71} (3) (1990), 275--287.

\bibitem{Ledrappier}
F. Ledrappier, \emph{Propri\'et\'es ergodiques des mesures de Sinai},
Publ. math. de l'I.H.\'E.S. {\bf 59} (1984), 163--188.

\bibitem{LS82}
F. Ledrappier, J.-M. Strelcyn, \emph{A proof of the estimation from below in Pesin's entropy
formula}, Ergod. Th. Dyn. Sys. {\bf 2} (1982), 203--219.

\bibitem{LedrappierYoung1}
F. Ledrappier, L.-S. Young, \emph{The metric entropy of diffeomorphisms. Part I: Characterization of measures satisfying Pesin's entropy formula}, Ann. Math. {\bf 122} (3) (1985), 509--539.

\bibitem{LedrappierYoung2}
F. Ledrappier, L.-S. Young, \emph{The metric entropy of diffeomorphisms. Part I: Relations between entropy, exponents and dimension}, Ann. Math. {\bf 122} (3) (1985), 540--574.

\bibitem{Lee}
J.M. Lee, \emph{Introduction to Smooth Manifolds}, Grad. Texts in Math. 218, 
Springer, New York, 2003.

\bibitem{Levy}
Y. Levy, \emph{Ergodic properties of the Lozi map}, Springer Lecture Notes in Mathematics, 1109, Springer-Verlag, Berlin (1985) 103--116.

\bibitem{LMM86}
R. de la Llave, J. M. Marco, R.  Moriyon, \emph{Canonical perturbation theory of Anosov systems and regularity results for the Livsic cohomology equation}, Ann.  Math. {\bf  123 }, (1986),  537--611.

\bibitem{Lozi}
R. Lozi, \emph{Un  attracteur \'etrange du type attracteur de H\'enon}, J. Phys., Paris {\bf 39}, Coll. C5 (1978), 9--10.

\bibitem{Misiurewicz}
M. Misiurewicz, \emph{Strange attractors for the Lozi mappings}, Nonlinear Dynamics, R.G. Helleman, ed. Academic: New York (1980) 348--358.

\bibitem{Pesin77}
Ya. Pesin, \emph{Lyapunov characteristic exponents and smooth ergodic theory}, Russ. Math. Surv. {\bf 32} (1977), no. 4, 55--114.

\bibitem{Pesin92} 
Ya. Pesin, \emph{Dynamical systems with generalized hyperbolic attractors: hyperbolic, ergodic and topological properties}, Ergod. Theory Dyn. Syst. {\bf 12}(1) (1992), 123--151. 

\bibitem{PesinChicago}
Ya. Pesin, \emph{Dimension Theory in Dynamical Systems: Contemporary Views and Applications}, Chicago Lectures in Math. Series,
The Chicago Univ. Press, Chicago, 1997.

\bibitem{Pesin1} Ya. Pesin, \emph{Lectures on partial hyperbolicity and stable ergodicity}, Z\"{u}rich Lectures in Advanced Mathematics, European Mathematical Society (EMS), Z\"{u}rich, (2004).

\bibitem{Pesin}
Ya. Pesin, \textit{Sinai's work on Markov partitions and SRB measures}, The Abel Prize Laureates book series, Springer-Verlag, 2018.

\bibitem{PesinSinai}
Ya. Pesin, Ya. Sinai, \emph{Gibbs measures for partially hyperbolic attractors},
Ergod. Th. Dyn. Sys. {\bf 2} (1982), 417--438.


\bibitem{Plykin} 
R. Plykin, \emph{Sources and sinks of A-diffeomorphisms of surfaces}, Matem. Sbornik {\bf 23} (1974), 8--13.

\bibitem{Poll93}
M. Pollicott, \emph{Lectures on ergodic theory and Pesin theory on compact manifolds}, London Mathematical Society Lecture Note Series, vol. 180, Cambridge University Press, Cambridge, 1993.

\bibitem{Putnam}
I.F. Putnam, \emph{A homology theory for Smale spaces}, Memoirs Amer. Math. Soc. vol. 232 no. 1094, Amer. Math. Soc.,  Providence, 2014.

\bibitem{RS80}
M. Reed, B. Simon, \emph{Methods of Modern Mathematical Physics I: Functional Analysis}, Academic Press, San Diego, 1980.

\bibitem{Robinson}
C. Robinson, \emph{Dynamical Systems, Stability, Symbolic Dynamics, and Chaos}, second edition, CRC Press, Boca Raton, FL, 1999.

\bibitem{RogersTeplyaev10}
L. Rogers, A. Teplyaev, \emph{Laplacians on the basilica Julia set}, Comm. Pure Appl. Anal. {\bf 9}(2010), 211--231. 


\bibitem{Rokhlin} 
V.A. Rokhlin, \emph{On the fundamental ideas of measure theory}, Transl. Amer. Math. Soc., Series 1, {\bf 10} (1962), 1--52.

\bibitem{Ruelle76}
D. Ruelle, \emph{A Measure Associated with Axiom-A Attractors}, Amer. J. Math. {\bf 98}(3) (1976), 619--654. 

\bibitem{Sataev99} 
E.A. Sataev, \emph{Ergodic properties of the Belykh map}, J. Math. Sci. {\bf 95} (1999), 2564--2575.

\bibitem{Shub}
M. Shub, \emph{Global Stability of Dynamical Systems}, Springer, Berlin, 1987.

\bibitem{Sinai68}
Ya. Sinai, \emph{Markov Partitions and C-diffeomorphisms}, Funct. Anal. and Appl.
{\bf 2} (1)(1968) 64--89. 

\bibitem{Sinai72}
Ya. Sinai, \emph{Gibbs measures in ergodic theory}, Russian Math. Surveys, {\bf 27}(4)
(1972), 21--69.


\bibitem{Steinhurst}
B. Steinhurst, \emph{Uniqueness of Locally Symmetric Brownian Motion on Laakso Spaces}, Pot. Anal. {\bf 38} (2013), 281--298.

\bibitem{Tak}
M. Takesaki, \emph{Theory of Operator Algebras I},
Encycl. Math. Sci. 124, Springer, New York, 2002.


\bibitem{TrutnauShin}
G. Trutnau, J. Shin, \emph{On the stochastic regularity of distorted Brownian motions},
Trans. Amer. Math. Soc. {\bf 369} (2017), 7883--7915.

\bibitem{Viana16} 
M. Viana, K. Oliveira, \emph{Foundations of Ergodic Theory}, 
Cambridge University Press, Cambridge, 2016.

\bibitem{Walters}
P. Walters, \emph{An Introduction to Ergodic Theory},  Grad. Texts in Math. 79, Springer, New York, 2000.

\bibitem{Williams}
R.F. Williams, \emph{Expanding attractors}, Publ. math. de l'I.H.\'E.S., tome {\bf 43} (1974), 169--203.

\bibitem{Young}
L.-S. Young, \emph{What Are SRB Measures, and Which Dynamical Systems Have Them ?},
J. Stat. Phys., {\bf 108} (2002),  no. 5-6, 733--754.


\bibitem{Yue} 
C. B. Yue, \emph{Integral formulas for the Laplacian along the unstable foliation and applications to rigidity problems for manifolds of negative curvature}, Ergod. Th. Dyn. Sys. {\bf 11} (4) (1991), 803--819.

\bibitem{Yue95} 
C. B. Yue, \emph{Brownian motion on Anosov foliations and manifolds
of negative curvature}, J. Diff. Geom. {\bf 41} (1995), 159--183.


\end{thebibliography}
\end{document}